\documentclass{amsart}%
\usepackage{amsmath}
\usepackage{color}
\usepackage{amssymb}
\usepackage{amsfonts}
\usepackage{graphicx}%
\newtheorem{theorem}{Theorem}[section]
\theoremstyle{plain}
\newtheorem{acknowledgement}{Acknowledgement}

\newtheorem{example}[theorem]{Examples}

\newtheorem{lemma}[theorem]{Lemma}

\newtheorem{proposition}[theorem]{Proposition}

\numberwithin{equation}{section}

\def \w{\mathcal W}

\begin{document}
\title[Entire solutions to a critical competitive system]{Existence and phase separation of entire solutions to a pure critical
competitive elliptic system}
\author{M\'{o}nica Clapp}
\address{Instituto de Matem\'{a}ticas, Universidad Nacional Aut\'{o}noma de M\'{e}xico,
Circuito Exterior, C.U., 04510 Coyoac\'{a}n, CDMX, Mexico}
\email{monica.clapp@im.unam.mx}
\author{Angela Pistoia}
\address{Dipartimento di Metodi e Modelli Matematici, Universit\`{a} di Roma
\textquotedblleft La Sapienza\textquotedblright, Roma, Italy}
\email{angela.pistoia@uniroma1.it}
\thanks{M. Clapp was partially supported by CONACYT grant 237661 (Mexico) and
UNAM-DGAPA-PAPIIT grant IN104315 (Mexico). A. Pistoia was partially supported
by Fondi di Ateneo ``Sapienza'' Universit\'a di Roma (Italy).}
\date{\today}

\begin{abstract}
We establish the existence of a positive fully nontrivial solution $(u,v)$ to
the weakly coupled elliptic system%
\[
\left\{
\begin{tabular}
[c]{l}%
$-\Delta u=\mu_{1}|u|^{{2}^{\ast}-2}u+\lambda\alpha|u|^{\alpha-2}|v|^{\beta
}u,$\\
$-\Delta v=\mu_{2}|v|^{{2}^{\ast}-2}v+\lambda\beta|u|^{\alpha}|v|^{\beta{-2}%
}v,$\\
$u,v\in D^{1,2}(\mathbb{R}^{N}),$%
\end{tabular}
\ \right.
\]
where $N\geq4,$ $2^{\ast}:=\frac{2N}{N-2}$ is the critical Sobolev exponent,
$\alpha,\beta\in(1,2],$ $\alpha+\beta=2^{\ast},$ $\mu_{1},\mu_{2}>0,$ and
$\lambda<0.$ We show that these solutions exhibit phase separation as
$\lambda\rightarrow-\infty,$ and we give a precise description of their limit domains.

If $\mu_{1}=\mu_{2}$ and $\alpha=\beta$, we prove that the system has
infinitely many fully nontrivial solutions, which are not conformally equivalent.\smallskip\ 



\end{abstract}

\maketitle

\baselineskip15pt

\section{Introduction}

We study the weakly coupled elliptic system
\begin{equation}
\left\{
\begin{tabular}
[c]{l}%
$-\Delta u=\mu_{1}|u|^{{2}^{\ast}-2}u+\lambda\alpha|u|^{\alpha-2}|v|^{\beta
}u,$\\
$-\Delta v=\mu_{2}|v|^{{2}^{\ast}-2}v+\lambda\beta|u|^{\alpha}|v|^{\beta{-2}%
}v,$\\
$u,v\in D^{1,2}(\mathbb{R}^{N}),$%
\end{tabular}
\ \right.  \label{system}%
\end{equation}
where $N\geq4,$ $2^{\ast}:=\frac{2N}{N-2}$ is the critical Sobolev exponent,
$\alpha,\beta\in(1,2],$ $\alpha+\beta=2^{\ast},$ $\mu_{1},\mu_{2}>0,$ and
$\lambda\in\mathbb{R}$.

The solutions to this system are solitary waves for a system of coupled
Gross-Pitaevskii equations. This type of systems arises, e.g., in the
Hartree-Fock theory for double condensates, that is, Bose-Einstein condensates
of two different hyperfine states which overlap in space; see \cite{egbb}. The
sign of $\mu_{i}$ reflects the interaction of the particles within each single
state. If $\mu_{i}$ is positive, this interaction is attractive. The sign of
$\lambda,$ on the other hand, reflects the interaction of particles in
different states. This interaction is attractive if $\lambda>0$ and it is
repulsive if $\lambda<0.$ If the condensates repel, they separate spatially.
This phenomenon is called phase separation and has been described in \cite{t}.

Motivated by their physical applications, weakly coupled elliptic systems have
received much attention in recent years, and there are many results for the
cubic case - where $\alpha=\beta=2$ and $2^{\ast}$ is replaced by $4$ - in low
dimensions $N\leq3$; see, e.g.,
\cite{ac,bdw,bww,dww,lw1,lw2,liw,mmp,si,so,sot,ww}. In this case, the
nonlinear terms are subcritical.

In contrast, there are only few results for the critical case. For a Brezis-Nirenberg type system in a bounded domain of dimension $N\geq4$ existence results were recently obtained by Chen and Zhou in \cite{cz1,cz2}. They also exhibited phase separation for $N\geq6.$ An unbounded sequence of sign-changing solutions for $N\geq7$ and $\alpha=\beta$ was obtained in \cite{llw}, and spiked solutions were constructed in \cite{pt} for $N=4.$ Some existence and multiplicity results for a Coron type system in a bounded domain with one or multiple small holes were recently obtained in \cite{ppw,ps}.

We are interested in solutions to the system (\ref{system}) in the whole space
$\mathbb{R}^{N}$. When $\lambda=0$ this system reduces to the single equation%
\begin{equation}
-\Delta w=|w|^{{2}^{\ast}-2}w,\text{\qquad}w\in D^{1,2}(\mathbb{R}^{N}).
\label{bc}%
\end{equation}
It is well known that the problem (\ref{bc}) has a positive solution and
infinitely many sign-changing solutions. Note that, if $w$ solves (\ref{bc}),
then $u=\mu_{1}^{\frac{2-N}{4}}w,$ $v=0,$ and $u=0,$ $v=\mu_{2}^{\frac{2-N}%
{4}}w,$ solve (\ref{system}). So the system has infinitely many solutions with
one trivial component. We are interested in solutions where both components,
$u$ and $v,$ are nontrivial. They are called \emph{fully nontrivial}
solutions. A solution is said to be \emph{positive} if $u\geq0$ and $v\geq0$,
and it is said to be \emph{synchronized} if it is of the form $(su,tu)$ with
$s,t\in\mathbb{R}$.

In the cooperative case, i.e., when $\lambda>0,$ Chen and Zou established the
existence of a positive least energy fully nontrivial solution to the system
(\ref{system}) with $\alpha=\beta=\frac{2^{\ast}}{2}$ for all $\lambda>0$ if
$N\geq5$ and for a wide range of $\lambda>0$ if $N=4;$ see \cite{cz1,cz2}.
Peng, Peng and Wang \cite{ppw} studied the system for $\mu_{1}=\mu_{2}=1,$
$\lambda=\frac{1}{2^{\ast}}$ and different values of $\alpha$ and $\beta,$ and
they obtained uniqueness and nondegeneracy results for positive synchronized solutions. Guo, Li and Wei studied the critical system (\ref{system}) in dimension $N=3$ for $\lambda <0$ and they established the existence of positive solutions with $k$ peaks for $k$ sufficiently large in \cite{glw}. In  \cite{ggt1, ggt2} Gladiali, Grossi and Troestler obtained radial and nonradial solutions to some critical systems using bifurcation methods. 

Here we focus our attention to the competitive case, i.e., to $\lambda<0.$ In
this case, the system (\ref{system}) does not have a least energy fully
nontrivial solution; see Proposition \ref{prop:nonexistence} below. This
behavior showcases the lack of compactness of the variational functional,
which comes from the fact that system is invariant under translations and
dilations, that allow functions to travel to infinity and to blow up without
changing their energy value.

But the conformal invariance of the system (\ref{system}) can also be used to
our advantage. There are groups of conformal transformations of $\mathbb{R}%
^{N}$ which have the property that all of their orbits have positive
dimension. So, as blow-up can only occur at points, looking for solutions
which are invariant under such group actions will restore compactness. W. Ding
used this fact in \cite{d}\ to establish the existence of infinitely many
sign-changing solutions for the single equation (\ref{bc}). Note that linear
isometries of $\mathbb{R}^{N}$, on the other hand, do not serve this purpose,
because the origin is always a fixed point.

Let $O(N+1)$ be the group of linear isometries of $\mathbb{R}^{N+1}$ and let
$\Gamma$ be a closed subgroup of $O(N+1).$ We write $\Gamma p:=\{\gamma p:\gamma\in\Gamma\}$ for the $\Gamma$-orbit of a point $p\in\mathbb{S}^{N}$. We shall look for solutions to the
system (\ref{system}) which are invariant under the conformal action of
$\Gamma$ on $\mathbb{R}^{N}$ induced by the stereographic projection
$\sigma:\mathbb{S}^{N}\rightarrow\mathbb{R}^{N}\cup\{\infty\}.$ Namely, for
each $\gamma\in\Gamma,$ we consider the map $\widetilde{\gamma}:\mathbb{R}%
^{N}\rightarrow\mathbb{R}^{N}$ given by $\widetilde{\gamma}x:=(\sigma
\circ\gamma^{-1}\circ\sigma^{-1})(x),$ which is well defined except at a
single point. The reason for considering this action is that $O(N+1)$ contains
subgroups $\Gamma$, which do not act transitively on $\mathbb{S}^N$ (i.e., $\Gamma p \neq \mathbb{S}^{N}$ for every $p\in\mathbb{S}^{N}$), with the property that the $\Gamma$-orbit $\Gamma p$ of every point $p\in\mathbb{S}^{N}$ has positive dimension. We may take, for example, $\Gamma:=O(m)\times O(n)$ with
$m+n=N+1,$ $m,n\geq2.$ These were the groups considered by W. Ding in
\cite{d}; see Examples \ref{example}\ below.

A function $u\ $will be said to be $\Gamma$\textit{-invariant} if
\[
\left\vert \det\widetilde{\gamma}^{\prime}(x)\right\vert ^{1/2^{\ast}%
}u(\widetilde{\gamma}x)=u(x)\text{\qquad for all \ }\gamma\in\Gamma,\text{
}x\in\mathbb{R}^{N}\text{,}%
\]
and a pair of functions $(u,v)$ will be said to be $\Gamma$\emph{-invariant}
if each of them is $\Gamma$-invariant. We will prove the following results.

\begin{theorem}
\label{thm:main}Let $\Gamma$ be a closed subgroup of $O(N+1)$ such that $\Gamma$ does not act transitively on $\mathbb{S}^{N}$ and the
$\Gamma$-orbit of every point $p\in\mathbb{S}^{N}$ has positive dimension.
Then, the following statements hold true:

\begin{enumerate}
\item[(a)] The system \emph{(\ref{system})} has a positive fully nontrivial
$\Gamma$-invariant solution for each $\lambda<0$.

\item[(b)] If $\mu_{1}=\mu_{2}=:\mu$ and $\alpha=\beta,$ then, for each
$\lambda\leq-\frac{\mu}{\alpha},$ the system \emph{(\ref{system})} has
infinitely many fully nontrivial $\Gamma$-invariant solutions, which are not
conformally equivalent.

\item[(c)] There exists a $\lambda_{\ast}<0,$ which depends on $\mu_{1}%
,\mu_{2},\alpha,\beta,$ such that the system \emph{(\ref{system})} does not
have a fully nontrivial synchronized solution if $\lambda<\lambda_{\ast}.$
\end{enumerate}
\end{theorem}

The next result says that there is phase separation for the positive solutions.

\begin{theorem}
\label{thm:separation}Assume that $\Gamma$ does not act transitively on $\mathbb{S}^{N}$ and that the $\Gamma$-orbit of every point $p\in\mathbb{S}^{N}$ has positive dimension. For $\lambda_{k}<0$ with
$\lambda_{k}\rightarrow-\infty$ let $(u_{k},v_{k})$ be the positive fully
nontrivial $\Gamma$-invariant solution for the system \emph{(\ref{system})}
with $\lambda=\lambda_{k}$ given by \emph{Theorem \ref{thm:main}(a)}. Then,
after passing to a subsequence, we have that $u_{k}\rightarrow u_{\infty}$ and
$v_{k}\rightarrow v_{\infty}$ strongly in $D^{1,2}(\mathbb{R}^{N}),$ the
functions $u_{\infty}$ and $v_{\infty}\ $are continuous, $u_{\infty}\geq0,$
$v_{\infty}\geq0,\ u_{\infty}v_{\infty}\equiv0,$ $u_{\infty}$ solves the
problem%
\[
-\Delta u=\mu_{1}|u|^{{2}^{\ast}-2}u,\text{\qquad}u\in D_{0}^{1,2}(\Omega
_{1}),
\]
and $v_{\infty}$ solves the problem%
\[
-\Delta v=\mu_{2}|v|^{{2}^{\ast}-2}v,\qquad v\in D_{0}^{1,2}(\Omega_{2}),
\]
where $\Omega_{1}:=\{x\in\mathbb{R}^{N}:u_{\infty}(x)>0\}$ and $\Omega
_{2}:=\{x\in\mathbb{R}^{N}:v_{\infty}(x)>0\}.$ Moreover, $\Omega_{1}$ and $\Omega_{2}$ are $\Gamma$-invariant and connected, $\Omega_{1}\cap\Omega_{2}=\emptyset$ and $\overline{\Omega_{1}\cup\Omega_{2}}=\mathbb{R}^{N}.$
\end{theorem}

We wish to stress that Theorem \ref{thm:separation} gives very precise information on the domains $\Omega_{1}$ and $\Omega_{2}$, as the following result shows.

\begin{proposition}
\label{prop:tori}
Let $\Gamma:=O(m)\times O(n)$ with $m+n=N+1,$ $m,n\geq2$. Then, after adding a point at infinity and up to relabeling, the domains $\Omega_{1}$ and $\Omega_2$ given by \emph{Theorem \ref{thm:separation}} have the following shape: $\Omega_{1}$ is diffeomorphic to $\mathbb{S}^{m-1}\times \mathbb{B}^{n}$, $\Omega_{2}$ is diffeomorphic to $\mathbb{B}^{m}\times \mathbb{S}^{n-1}$, and their common boundary is diffeomorphic to $\mathbb{S}^{m-1}\times \mathbb{S}^{n-1}$, where $\mathbb{B}^{k}$ and $\mathbb{S}^{k-1}$ denote the open unit ball and the unit sphere in $\mathbb{R}^{k}$ respectively.
\end{proposition}

This paper is organized as follows. In Section \ref{sec:variational} we
discuss the variational setting and we prove part (c) of Theorem
\ref{thm:main}. Part (a) is proved in Section \ref{sec:existence}\ and part
(b) in Section \ref{sec:multiplicity}. Section \ref{sec:separation}\ is
devoted to the proof of Theorem \ref{thm:separation} and Proposition \ref{prop:tori}.

\section{The variational setting}

\label{sec:variational}Let $\mathbf{D}:=D^{1,2}(\mathbb{R}^{N})\times
D^{1,2}(\mathbb{R}^{N})$ where, as usual, $D^{1,2}(\mathbb{R}^{N}):=\{u\in
L^{2^{\ast}}(\mathbb{R}^{N}):\nabla u\in L^{2}(\mathbb{R}^{N},\mathbb{R}%
^{N})\}.$ The scalar product in $\mathbf{D}$ is given by%
\[
\left\langle (u_{1},v_{1}),(u_{2},v_{2})\right\rangle :=\int_{\mathbb{R}^{N}%
}\nabla u_{1}\cdot\nabla u_{2}+\nabla v_{1}\cdot\nabla v_{2}.
\]
The solutions to the system (\ref{system}) are the critical points of the
functional $E:\mathbf{D}\rightarrow\mathbb{R}$ defined by
\[
E\left(  u,v\right)  :=\frac{1}{2}\int_{\mathbb{R}^{N}}\left(  \left\vert
\nabla u\right\vert ^{2}+\left\vert \nabla v\right\vert ^{2}\right)  -\frac
{1}{2^{\ast}}\int_{\mathbb{R}^{N}}\left(  \mu_{1}|u|^{{2}^{\ast}}+\mu
_{2}|v|^{{2}^{\ast}}\right)  -\lambda\int_{\mathbb{R}^{N}}|u|^{{\alpha}%
}|v|^{{\beta}}.
\]
Note that, as $\alpha,\beta>1,$ this functional is of class $\mathcal{C}^{1}.$
We write%
\begin{align*}
f(u,v)  &  :=\partial_{u}E\left(  u,v\right)  u=\int_{\mathbb{R}^{N}%
}\left\vert \nabla u\right\vert ^{2}-\mu_{1}\int_{\mathbb{R}^{N}}%
|u|^{{2}^{\ast}}-\lambda\alpha\int_{\mathbb{R}^{N}}|u|^{{\alpha}}|v|^{{\beta}%
},\\
h(u,v)  &  :=\partial_{v}E\left(  u,v\right)  v=\int_{\mathbb{R}^{N}%
}\left\vert \nabla v\right\vert ^{2}-\mu_{2}\int_{\mathbb{R}^{N}}%
|v|^{{2}^{\ast}}-\lambda\beta\int_{\mathbb{R}^{N}}|u|^{{\alpha}}|v|^{{\beta}}.
\end{align*}
The fully nontrivial solutions to (\ref{system})\ lie on the set%
\[
\mathcal{N}:=\{(u,v)\in\mathbf{D}:u\neq0,\text{ }v\neq0,\text{ }%
f(u,v)=0,\text{ }h(u,v)=0\},
\]
which is called the \emph{Nehari manifold} and has the following properties.

\begin{proposition}
\label{prop:nehari}

\begin{enumerate}
\item[(a)] For every $(u,v)\in\mathcal{N}$, one has that%
\[
\mu_{1}^{-(N-2)/2}S^{N/2}\leq\int_{\mathbb{R}^{N}}\left\vert \nabla
u\right\vert ^{2}\text{,\qquad}\mu_{2}^{-(N-2)/2}S^{N/2}\leq\int
_{\mathbb{R}^{N}}\left\vert \nabla v\right\vert ^{2},
\]
where $S$ is the best constant for the embedding $D^{1,2}%
(\mathbb{R}^{N})\hookrightarrow L^{2^{\ast}}(\mathbb{R}^{N}).$

\item[(b)] $\mathcal{N}$ is a closed $\mathcal{C}^{1}$-submanifold of
codimension $2$ of the Hilbert space $\mathbf{D},$ and the tangent space to
$\mathcal{N}$\ at the point $(u,v)$ is the orthogonal complement in
$\mathbf{D}$ of the linear subspace generated by $\nabla f(u,v)$ and $\nabla
h(u,v).$

\item[(c)] $\mathcal{N}$ is a natural constraint for the functional $E,$ i.e.,
a critical point of the restriction of $E$ to $\mathcal{N}$ is a critical
point of $E.$

\item[(d)] If $(u,v)\in\mathcal{N}$, then $E(u,v)=\max\left\{
E(su,tv):s>0,\text{\thinspace}t>0\right\}  .$
\end{enumerate}
\end{proposition}

\begin{proof}
(a) \ Let $(u,v)\in\mathcal{N}$. Then, as $f(u,v)=0$, $h(u,v)=0$ and
$\lambda<0,$ we have that%
\[
\int_{\mathbb{R}^{N}}\left\vert \nabla u\right\vert ^{2}\leq\mu_{1}%
\int_{\mathbb{R}^{N}}|u|^{{2}^{\ast}}\text{\qquad and\qquad}\int
_{\mathbb{R}^{N}}\left\vert \nabla v\right\vert ^{2}\leq\mu_{2}\int
_{\mathbb{R}^{N}}\left\vert v\right\vert ^{{2}^{\ast}}.
\]
Since $u\neq0$ and $v\neq0,$ using the Sobolev inequality we get that%
\[
0<S\leq\frac{\int_{\mathbb{R}^{N}}\left\vert \nabla u\right\vert ^{2}}{\left(
\int_{\mathbb{R}^{N}}|u|^{{2}^{\ast}}\right)  ^{2/{2}^{\ast}}}\leq\mu
_{1}^{2/2^{\ast}}\left(  \int_{\mathbb{R}^{N}}\left\vert \nabla u\right\vert
^{2}\right)  ^{({2}^{\ast}-2)/{2}^{\ast}}%
\]
and
\[
0<S\leq\frac{\int_{\mathbb{R}^{N}}\left\vert \nabla v\right\vert ^{2}}{\left(
\int_{\mathbb{R}^{N}}|v|^{{2}^{\ast}}\right)  ^{2/{2}^{\ast}}}\leq\mu
_{2}^{2/2^{\ast}}\left(  \int_{\mathbb{R}^{N}}\left\vert \nabla v\right\vert
^{2}\right)  ^{({2}^{\ast}-2)/{2}^{\ast}}.
\]
This proves (a).

(b) \ Statement (a) implies that $\mathcal{N}$ is closed in $\mathbf{D}$. Next
we show that $\nabla f(u,v)$ and $\nabla h(u,v)$ are linearly independent for
every $(u,v)\in\mathcal{N}$. If $s\nabla f(u,v)+t\nabla h(u,v)=0$ for some
$(u,v)\in\mathcal{N},$ $s,t\in\mathbb{R}$, then%
\begin{align*}
&  0=s\left\langle \nabla f(u,v),(u,0)\right\rangle +t\left\langle \nabla
h(u,v),(u,0)\right\rangle \\
&  =s\left(  2\int_{\mathbb{R}^{N}}\left\vert \nabla u\right\vert ^{2}%
-2^{\ast}\mu_{1}\int_{\mathbb{R}^{N}}|u|^{{2}^{\ast}}-\lambda\alpha^{2}%
\int_{\mathbb{R}^{N}}|u|^{{\alpha}}|v|^{{\beta}}\right)  +t\left(
-\lambda\alpha\beta\int_{\mathbb{R}^{N}}|u|^{{\alpha}}|v|^{{\beta}}\right) \\
&  =s\left(  \left(  2-2^{\ast}\right)  \mu_{1}\int_{\mathbb{R}^{N}}%
|u|^{{2}^{\ast}}+\lambda\alpha(2-\alpha)\int_{\mathbb{R}^{N}}|u|^{{\alpha}%
}|v|^{{\beta}}\right)  +t\left(  -\lambda\alpha\beta\int_{\mathbb{R}^{N}%
}|u|^{{\alpha}}|v|^{{\beta}}\right) \\
&  =:sa_{11}+ta_{12},
\end{align*}
and%
\begin{align*}
&  0=s\left\langle \nabla f(u,v),(0,v)\right\rangle +t\left\langle \nabla
h(u,v),(0,v)\right\rangle \\
&  =s\left(  -\lambda\alpha\beta\int_{\mathbb{R}^{N}}|u|^{{\alpha}}%
|v|^{{\beta}}\right)  +t\left(  2\int_{\mathbb{R}^{N}}\left\vert \nabla
v\right\vert ^{2}-2^{\ast}\mu_{2}\int_{\mathbb{R}^{N}}|v|^{{2}^{\ast}}%
-\lambda\beta^{2}\int_{\mathbb{R}^{N}}|u|^{{\alpha}}|v|^{{\beta}}\right) \\
&  =s\left(  -\lambda\alpha\beta\int_{\mathbb{R}^{N}}|u|^{{\alpha}}%
|v|^{{\beta}}\right)  +t\left(  \left(  2-2^{\ast}\right)  \mu_{2}%
\int_{\mathbb{R}^{N}}|v|^{{2}^{\ast}}+\lambda\beta(2-\beta)\int_{\mathbb{R}%
^{N}}|u|^{{\alpha}}|v|^{{\beta}}\right) \\
&  =:sa_{21}+ta_{22}.
\end{align*}
If $\int_{\mathbb{R}^{N}}|u|^{{\alpha}}|v|^{{\beta}}=0$, statement (a) implies
that
\[
\det(a_{ij})\geq\left(  2-2^{\ast}\right)  ^{2}c_{0}^{2}>0,
\]
where $c_{0}:=\min\{\mu_{1}^{-(N-2)/2},\mu_{2}^{-(N-2)/2}\}S^{N/2}.$ If
$\int_{\mathbb{R}^{N}}|u|^{{\alpha}}|v|^{{\beta}}\neq0$ then, as $\alpha
,\beta\in(1,2]$ and $\lambda<0$, we have that%
\begin{align*}
A  &  :=\frac{\mu_{1}\int_{\mathbb{R}^{N}}|u|^{{2}^{\ast}}}{-\lambda
\int_{\mathbb{R}^{N}}|u|^{{\alpha}}|v|^{{\beta}}}\geq\frac{c_{0}}{-\lambda
\int_{\mathbb{R}^{N}}|u|^{{\alpha}}|v|^{{\beta}}}+\alpha>\left(  \frac{c_{0}%
}{-2\lambda\int_{\mathbb{R}^{N}}|u|^{{\alpha}}|v|^{{\beta}}}+1\right)
\alpha=:C\alpha,\\
B  &  :=\frac{\mu_{2}\int_{\mathbb{R}^{N}}|v|^{{2}^{\ast}}}{-\lambda
\int_{\mathbb{R}^{N}}|u|^{{\alpha}}|v|^{{\beta}}}\geq\frac{c_{0}}{-\lambda
\int_{\mathbb{R}^{N}}|u|^{{\alpha}}|v|^{{\beta}}}+\beta>\left(  \frac{c_{0}%
}{-2\lambda\int_{\mathbb{R}^{N}}|u|^{{\alpha}}|v|^{{\beta}}}+1\right)
\beta=:C\beta.
\end{align*}
We use these inequalities, and the fact that $\alpha,\beta\in(1,2]$ and
$\alpha+\beta=2^{\ast},$ to estimate the determinant%
\begin{align*}
&  \left\vert
\begin{array}
[c]{cc}%
\left(  2-2^{\ast}\right)  \frac{\mu_{1}\int_{\mathbb{R}^{N}}|u|^{{2}^{\ast}}%
}{-\lambda\int_{\mathbb{R}^{N}}|u|^{{\alpha}}|v|^{{\beta}}}-\alpha(2-\alpha) &
\alpha\beta\\
\alpha\beta & \left(  2-2^{\ast}\right)  \frac{\mu_{2}\int_{\mathbb{R}^{N}%
}|v|^{{2}^{\ast}}}{-\lambda\int_{\mathbb{R}^{N}}|u|^{{\alpha}}|v|^{{\beta}}%
}-\beta(2-\beta)
\end{array}
\right\vert \\
&  =\left(  2-2^{\ast}\right)  ^{2}AB-\left(  2-2^{\ast}\right)  \left(
\beta(2-\beta)A+\alpha(2-\alpha)B\right)  +\alpha\beta(2-\alpha)(2-\beta
)-(\alpha\beta)^{2}\\
&  \geq C\alpha\beta\left[  \left(  2-2^{\ast}\right)  ^{2}-(2-2^{\ast
})(4-2^{\ast})\right]  +\alpha\beta\left[  (2-\alpha)(2-\beta)-\alpha
\beta\right] \\
&  =\frac{\left(  2^{\ast}-2\right)  c_{0}\alpha\beta}{-\lambda\int
_{\mathbb{R}^{N}}|u|^{{\alpha}}|v|^{{\beta}}}.
\end{align*}
It follows that%
\begin{equation}
\det(a_{ij})\geq\left(  2^{\ast}-2\right)  c_{0}\alpha\beta\left(
-\lambda\right)  \int_{\mathbb{R}^{N}}|u|^{{\alpha}}|v|^{{\beta}}>0.
\label{eq:det}%
\end{equation}
Thus, in both cases, $s=t=0.$ This proves that $\nabla f(u,v)$ and $\nabla
h(u,v)$ are linearly independent for every $(u,v)\in\mathcal{N}$. Therefore,
$\mathcal{N}$ is a $\mathcal{C}^{1}$-submanifold of $\mathbf{D}$ and the
tangent space to $\mathcal{N}$\ at the point $(u,v)$ is the orthogonal
complement in $\mathbf{D}$ of the linear subspace generated by $\nabla f(u,v)$
and $\nabla h(u,v).$

(c) \ If $(u,v)\in\mathcal{N}$ is a critical point of the restriction of $E$
to $\mathcal{N}$, then $\nabla E(u,v)=s\nabla f(u,v)+t\nabla h(u,v)$ for some
$s,t\in\mathbb{R}$. Taking the scalar product with $(u,0)$ and $(0,v)$ we get
that%
\begin{align*}
s\left\langle \nabla f(u,v),(u,0)\right\rangle +t\left\langle \nabla
h(u,v),(u,0)\right\rangle  &  =\left\langle \nabla E(u,v),(u,0)\right\rangle
=f(u,v)=0,\\
s\left\langle \nabla f(u,v),(0,v)\right\rangle +t\left\langle \nabla
h(u,v),(0,v)\right\rangle  &  =\left\langle \nabla E(u,v),(0,v)\right\rangle
=h(u,v)=0.
\end{align*}
But we have already shown that this implies that $s=t=0.$ Hence, $\nabla
E(u,v)=0,$ i.e., $(u,v)$ is a critical point of $E.$

(d) \ Fix $(u,v)\in\mathcal{N}$ and let $(\hat{s},\hat{t})$ be a critical
point of the function $e(s,t):=E(su,tv)$ in $(0,\infty)\times(0,\infty).$
Then, as $s\frac{\partial e}{\partial s}(s,t)=f(su,tv)$ and $t\frac{\partial
e}{\partial t}(s,t)=h(su,tv),$ we have that $(\hat{s}u,\hat{t}v)\in
\mathcal{N}$. Moreover,%
\begin{align*}
\hat{s}^{2}\frac{\partial^{2}e}{\partial s^{2}}(\hat{s},\hat{t})  &  =\left(
2-2^{\ast}\right)  \mu_{1}\int_{\mathbb{R}^{N}}|\hat{s}u|^{{2}^{\ast}}%
+\lambda\alpha(2-\alpha)\int_{\mathbb{R}^{N}}|\hat{s}u|^{{\alpha}}|\hat
{t}v|^{{\beta}},\\
\hat{t}^{2}\frac{\partial^{2}e}{\partial t^{2}}(\hat{s},\hat{t})  &  =\left(
2-2^{\ast}\right)  \mu_{2}\int_{\mathbb{R}^{N}}|\hat{t}v|^{{2}^{\ast}}%
+\lambda\beta(2-\beta)\int_{\mathbb{R}^{N}}|\hat{s}u|^{{\alpha}}|\hat
{t}v|^{{\beta}},\\
\hat{s}\hat{t}\frac{\partial^{2}e}{\partial t\partial s}(\hat{s},\hat{t})  &
=-\lambda\alpha\beta\int_{\mathbb{R}^{N}}|\hat{s}u|^{{\alpha}}|\hat
{t}v|^{{\beta}}.
\end{align*}
Hence, $\frac{\partial^{2}e}{\partial s^{2}}(\hat{s},\hat{t})<0,$
$\frac{\partial^{2}e}{\partial t^{2}}(\hat{s},\hat{t})<0$ and, as shown in
part (b),
\[
\left(  \hat{s}\hat{t}\right)  ^{2}\left(  \frac{\partial^{2}e}{\partial
s^{2}}(\hat{s},\hat{t})\frac{\partial^{2}e}{\partial t^{2}}(\hat{s},\hat
{t})-\left(  \frac{\partial^{2}e}{\partial t\partial s}(\hat{s},\hat
{t})\right)  ^{2}\right)  >0.
\]
Therefore, $(\hat{s},\hat{t})$ is a strict local maximum of $e.$ This implies
that $(1,1)$ is the only critical point of $e$ in $(0,\infty)\times(0,\infty)$
and it is a global maximum; see Lemma \ref{lem:unique} in the appendix.
\end{proof}

The following statement was proved in \cite{cz1,cz2}\ for $\alpha=\beta
=\frac{2^{\ast}}{2}.$ We give a simpler proof which applies to all
$\alpha,\beta.$

\begin{proposition}
\label{prop:nonexistence}$\inf_{(u,v)\in\mathcal{N}}E(u,v)=\frac{1}{N}(\mu
_{1}^{-(N-2)/2}+\mu_{2}^{-(N-2)/2})S^{N/2}$ and this value is not attained by
$E$ on $\mathcal{N}$.
\end{proposition}

\begin{proof}
If\ $(u,v)\in\mathcal{N}$,\ then Proposition \ref{prop:nehari}(a) yields%
\begin{align*}
E(u,v)  &  =E(u,v)-\frac{1}{2^{\ast}}E^{\prime}(u,v)\left[  (u,v)\right] \\
&  =\frac{1}{N}\int_{\mathbb{R}^{N}}\left(  \left\vert \nabla u\right\vert
^{2}+\left\vert \nabla v\right\vert ^{2}\right)  \geq\frac{1}{N}(\mu
_{1}^{-(N-2)/2}+\mu_{2}^{-(N-2)/2})S^{N/2}.
\end{align*}
To prove the opposite inequality, we choose a sequence of functions $w_{k}%
\in\mathcal{C}_{c}^{\infty}(B_{1}(0))$ in the unit ball $B_{1}(0):=\{x\in
\mathbb{R}^{N}:\left\vert x\right\vert <1\}$ which satisfies%
\[
\int_{B_{1}(0)}\left\vert \nabla w_{k}\right\vert ^{2}=\int_{B_{1}%
(0)}\left\vert w_{k}\right\vert ^{2^{\ast}}\text{\qquad and\qquad}\int
_{B_{1}(0)}\left\vert \nabla w_{k}\right\vert ^{2}\rightarrow S^{N/2}.
\]
Such a sequence exists because
\[
S=\inf_{\substack{w\in D_{0}^{1,2}\left(  \Omega\right)  \\w\neq0}}\frac
{\int_{\Omega}\left\vert \nabla w\right\vert ^{2}}{\left(  \int_{\Omega
}|w|^{{2}^{\ast}}\right)  ^{2/{2}^{\ast}}}%
\]
for every domain $\Omega$ in $\mathbb{R}^{N}$; see, e.g., \cite{s}. Fix
$\xi\in\mathbb{R}^{N}$ with $\left\vert \xi\right\vert =1$ and define
$u_{k}(x):=\mu_{1}^{-(N-2)/4}w_{k}(x-\xi)$ and $v_{k}(x):=\mu_{2}%
^{-(N-2)/4}w_{k}(x+\xi).$ As $u_{k}$ and $v_{k}$ have disjoint supports, we
have that $\int_{\mathbb{R}^{N}}|u_{k}|^{{\alpha}}|v_{k}|^{{\beta}}=0.$
Hence,
\begin{align*}
f(u_{k},v_{k})  &  =\mu_{1}^{(2-N)/2}\left(  \int_{B_{1}(0)}\left\vert \nabla
w_{k}\right\vert ^{2}-\int_{B_{1}(0)}\left\vert w_{k}\right\vert ^{2^{\ast}%
}\right)  =0,\\
h(u_{k},v_{k})  &  =\mu_{2}^{(2-N)/2}\left(  \int_{B_{1}(0)}\left\vert \nabla
w_{k}\right\vert ^{2}-\int_{B_{1}(0)}\left\vert w_{k}\right\vert ^{2^{\ast}%
}\right)  =0,
\end{align*}
i.e., $(u_{k},v_{k})\in\mathcal{N}$, and
\begin{align*}
E\left(  u_{k},v_{k}\right)   &  =\frac{1}{N}\left(  \mu_{1}^{-(N-2)/2}%
\int_{B_{1}(0)}\left\vert \nabla w_{k}\right\vert ^{2}+\mu_{2}^{-(N-2)/2}%
\int_{B_{1}(0)}\left\vert \nabla w_{k}\right\vert ^{2}\right) \\
&  \longrightarrow\frac{1}{N}(\mu_{1}^{-(N-2)/2}+\mu_{2}^{-(N-2)/2})S^{N/2}.
\end{align*}
This proves that $\inf_{(u,v)\in\mathcal{N}}E(u,v)=\frac{1}{N}(\mu
_{1}^{-(N-2)/2}+\mu_{2}^{-(N-2)/2})S^{N/2}.$

To show that this value is not attained, we argue by contradiction. Assume
that $(u_{0},v_{0})\in\mathcal{N}$ is a minimum of $E$ on $\mathcal{N}$. As
$(\left\vert u_{0}\right\vert ,\left\vert v_{0}\right\vert )\in\mathcal{N}$
and $E(\left\vert u_{0}\right\vert ,\left\vert v_{0}\right\vert )=E(u_{0}%
,v_{0}),$ the pair $(\left\vert u_{0}\right\vert ,\left\vert v_{0}\right\vert
)$ is also a minimum of $E$. So, we may assume that $u_{0}\geq 0$ and $v_{0}\geq 0.$ We consider two cases. If
$\int_{\mathbb{R}^{N}}u_{0}^{\alpha}v_{0}^{\beta} = 0$,  then $u_{0}^{\alpha}v_{0}^{\beta} = 0$ a.e. in $\mathbb{R}^N$ and $u_{0}$ solves the equation $-\Delta u = \mu_1|u|^{2^*-2}u$. As $v_0$ is nontrivial, we have that $u_0 = 0$ in a set of positive measure. This is a contradiction. If, on the other hand, $\int_{\mathbb{R}^{N}}u_{0}^{\alpha}v_{0}^{\beta} > 0$, then
\[
\int_{\mathbb{R}^{N}}\left\vert \nabla u_{0}\right\vert ^{2}<\mu_{1}%
\int_{\mathbb{R}^{N}}|u_{0}|^{{2}^{\ast}}\text{\qquad and\qquad}%
\int_{\mathbb{R}^{N}}\left\vert \nabla v_{0}\right\vert ^{2}<\mu_{2}%
\int_{\mathbb{R}^{N}}\left\vert v_{0}\right\vert ^{{2}^{\ast}},
\]
and from the Sobolev inequality we derive
\[
\mu_{1}^{-(N-2)/2}S^{N/2}<\int_{\mathbb{R}^{N}}\left\vert \nabla
u_{0}\right\vert ^{2}\text{\qquad and\qquad}\mu_{2}^{-(N-2)/2}S^{N/2}%
<\int_{\mathbb{R}^{N}}\left\vert \nabla v_{0}\right\vert ^{2}.
\]
This implies that $E(u_{0},v_{0})>\frac{1}{N}(\mu_{1}^{-(N-2)/2}+\mu_{2}%
^{-(N-2)/2})S^{N/2}$, which is, again, a contradiction.
\end{proof}

\begin{proposition}
\label{prop:nonsyncronized}There exists a $\lambda_{\ast}<0,$ depending on
$\mu_{1},\mu_{2},\alpha,\beta,$ such that%
\[
\mathcal{N}\cap\{(su,tu):s,t\in\mathbb{R},\text{ }u\in D^{1,2}(\mathbb{R}%
^{N})\}=\emptyset\qquad\text{if \ }\lambda<\lambda_{\ast}.
\]

\end{proposition}

\begin{proof}
To highlight the role of $\lambda,$ we write $\mathcal{N}_{\lambda},$
$f_{\lambda}$ and $h_{\lambda},$ instead of $\mathcal{N}$, $f$ and $h.$

Arguing by contradiction, assume there exists a sequence $(\lambda_{k})$ with
$\lambda_{k}\rightarrow-\infty$ for which there are $s_{k},t_{k}\in\mathbb{R}$
and $u_{k}\in D^{1,2}(\mathbb{R}^{N})$ such that $(s_{k}u_{k},t_{k}u_{k}%
)\in\mathcal{N}_{\lambda_{k}}.$ Then $s_{k}\neq0,$ $t_{k}\neq0$ and $u_{k}%
\neq0.$ So, after replacing $u_{k}$ with $r_{k}u_{k}$ for some suitable
$r_{k}>0,$ we may assume that $\int_{\mathbb{R}^{N}}\left\vert \nabla
u_{k}\right\vert ^{2}=\int_{\mathbb{R}^{N}}\left\vert u_{k}\right\vert
^{2^{\ast}}.$ We may also assume that $s_{k}>0,$ $t_{k}>0$. Then, dividing the
equations $f_{\lambda_{k}}(s_{k}u_{k},t_{k}u_{k})=0$ and $h_{\lambda_{k}%
}(s_{k}u_{k},t_{k}u_{k})=0$ by $\int_{\mathbb{R}^{N}}\left\vert u_{k}%
\right\vert ^{2^{\ast}},$ we obtain that $(s_{k},t_{k})$ solves the system
\[
\left\{
\begin{array}
[c]{c}%
1=\mu_{1}s_{k}^{2^{\ast}-2}+\lambda_{k}\alpha s_{k}^{\alpha-2}t_{k}^{\beta
},\medskip\\
1=\mu_{2}t_{k}^{2^{\ast}-2}+\lambda_{k}\beta s_{k}^{\alpha}t_{k}^{\beta-2}.
\end{array}
\right.
\]
Recall that $\alpha+\beta=2^{\ast}.$ Dividing the first equation by
$s_{k}^{\alpha-2}t_{k}^{\beta}$ and the second one by $s_{k}^{\alpha}%
t_{k}^{\beta-2}$ we get that%
\[%
\begin{array}
[c]{c}%
\mu_{1}\left(  \frac{s_{k}}{t_{k}}\right)  ^{\beta}=\frac{1}{s_{k}^{\alpha
-2}t_{k}^{\beta}}-\lambda_{k}\alpha\geq-\lambda_{k}\alpha,\medskip\\
\mu_{2}\left(  \frac{t_{k}}{s_{k}}\right)  ^{\alpha}=\frac{1}{s_{k}^{\alpha
}t_{k}^{\beta-2}}-\lambda_{k}\beta\geq-\lambda_{k}\beta.
\end{array}
\]
It follows that both sequences $(\frac{s_{k}}{t_{k}})$ and $(\frac{t_{k}%
}{s_{k}})$ are unbounded. This is a contradiction.
\end{proof}

\section{Symmetries and compactness}

\label{sec:existence}Let $(\mathbb{S}^{N},g)$ be the standard sphere and
$q\in\mathbb{S}^{N}$ be the north pole. The stereographic projection
$\sigma:\mathbb{S}^{N}\smallsetminus\{q\}\rightarrow\mathbb{R}^{N}$ is a
conformal diffeomorphism. The coordinates of the standard metric $g$ in the
chart given by $\sigma^{-1}:\mathbb{R}^{N}\rightarrow\mathbb{S}^{N}%
\smallsetminus\{q\}$ are $g_{ij}=\psi^{2^{\ast}-2}\delta_{ij},$ where
\[
\psi(x):=\left(  \frac{2}{1+\left\vert x\right\vert ^{2}}\right)
^{(N-2)/2},\text{\qquad}x\in\mathbb{R}^{N}.
\]
For $\mathfrak{u}\in\mathcal{C}^{\infty}(\mathbb{S}^{N})$, we set $u:=\psi
(\mathfrak{u}\circ\sigma^{-1})$ and we write $\nabla_{g}\mathfrak{u}$ for its gradient.

\begin{lemma}
\label{lem:invariance}
For every $\mathfrak{u},\mathfrak{v}\in\mathcal{C}^{\infty}(\mathbb{S}^{N})$ we have that
\begin{align*}
\int_{\mathbb{S}^{N}}\left(  |\nabla_{g}\mathfrak{u}|_{g}^{2}+\frac{N(N-2)}%
{4}\mathfrak{u}^{2}\right)  \mathrm{d}V_{g}  &  =\int_{\mathbb{R}^{N}%
}\left\vert \nabla u\right\vert ^{2}\mathrm{d}x,\\
\int_{\mathbb{S}^{N}}\left\vert \mathfrak{u}\right\vert ^{2^{\ast}}%
\mathrm{d}V_{g}  &  =\int_{\mathbb{R}^{N}}\left\vert u\right\vert ^{2^{\ast}%
}\mathrm{d}x,\\
\int_{\mathbb{S}^{N}}\left\vert \mathfrak{u}\right\vert ^{\alpha}\left\vert
\mathfrak{v}\right\vert ^{\beta}\mathrm{d}V_{g}  &  =\int_{\mathbb{R}^{N}%
}\left\vert u\right\vert ^{\alpha}\left\vert v\right\vert ^{\beta}%
\mathrm{d}x.
\end{align*}
\end{lemma}

\begin{proof}
If $(M,h)$ is a Riemannian manifold of dimension $n\geq 3$, the operator $L_h:=-\Delta_h + \frac{n-2}{4(n-1)}R_h$, where $\Delta_h := \mathrm{div}_h\nabla_h$ is the Laplace-Beltrami operator (without a sign) and $R_h$ is the scalar curvature with respect to the metric $h$, is called the conformal Laplacian. It has a certain conformal invariance, which in our case is expressed by the identity
$$-\Delta_g \mathfrak{u} + \frac{N(N-2)}{4}\mathfrak{u} = -\psi^{1-2^*}\Delta u;$$
see, e.g., \cite[Proposition 6.1.1]{h}. Note that the Riemannian volume element on $(\mathbb{S}^{N},g)$ is $\mathrm{d}V_{g}=\sqrt{\det(g_{ij})}\mathrm{d}x=\psi^{2^{\ast}}\mathrm{d}x$. So, multiplying this identity by $\mathfrak{u}$ and integrating by parts, we obtain
\begin{align*}
\int_{\mathbb{S}^{N}}\left(|\nabla_{g}\mathfrak{u}|_{g}^{2}+\frac{N(N-2)}{4}\mathfrak{u}^{2}\right)\mathrm{d}V_{g} &= \int_{\mathbb{S}^{N}}\left(-(\Delta_g \mathfrak{u})\mathfrak{u} + \frac{N(N-2)}{4}\mathfrak{u}^2 \right)\mathrm{d}V_{g}\\
& =\int_{\mathbb{R}^{N}}-(\Delta u)u\mathrm{d}x =\int_{\mathbb{R}^{N}}|\nabla u|^{2}\mathrm{d}x.
\end{align*}
This is the first identity in the statement of the lemma. The other two are immediate.
\end{proof}

Taking $\left(  \int_{\mathbb{S}^{N}}(|\nabla_{g}\mathfrak{u}|_{g}%
^{2}+\frac{N(N-2)}{4}\mathfrak{u}^{2})\mathrm{d}V_{g}\right)  ^{1/2}$ as the
norm in $H_{g}^{1}(\mathbb{S}^{N}),$ we obtain a linear isometry of Hilbert
spaces $\iota:H_{g}^{1}(\mathbb{S}^{N})\rightarrow D^{1,2}(\mathbb{R}^{N})$
given by%
\begin{equation}
\iota(\mathfrak{u}):=\psi(\mathfrak{u}\circ\sigma^{-1}).
\label{eq:stereographic}%
\end{equation}

$\mathbb{S}^{N}$ is invariant under the action of the group $O(N+1)$ of linear
isometries of $\mathbb{R}^{N+1},$ so each $\gamma\in O(N+1)$ induces a linear
isometry $\gamma:H_{g}^{1}(\mathbb{S}^{N})\rightarrow H_{g}^{1}(\mathbb{S}%
^{N})$ given by%
\[
(\gamma\mathfrak{u})(p):=\mathfrak{u}(\gamma^{-1}p),\text{\qquad}%
p\in\mathbb{S}^{N},\text{ \ }\mathfrak{u}\in H_{g}^{1}(\mathbb{S}^{N}).
\]
Therefore, the composition $\iota\circ\gamma\circ\iota^{-1}:D^{1,2}%
(\mathbb{R}^{N})\rightarrow D^{1,2}(\mathbb{R}^{N})$ is a linear isometry.
This gives an action of $O(N+1)$ on $\mathbf{D},$ defined by $\gamma
(u,v):=(\gamma u,\gamma v),$ where%
\[
\gamma u:=\left(  \iota\circ\gamma\circ\iota^{-1}\right)  u,\text{\qquad
}\gamma\in O(N+1),\text{ \ }u\in D^{1,2}(\mathbb{R}^{N}).
\]
Set $\widetilde{\gamma}:=\sigma\circ\gamma^{-1}\circ\sigma^{-1}.$ As $\sigma^{-1}$ is a conformal map and $\gamma^{-1}$ is a linear isometry, we have that $|\det\mathrm{d}_x \sigma^{-1}|=\psi^{2^*}(x)$ and $|\det\mathrm{d}_p\gamma^{-1}|=1$. Therefore,
$$|\det\widetilde{\gamma}'(x)| = |\det\mathrm{d}_{(\gamma^{-1}\circ\sigma^{-1})(x)}\sigma||\det\mathrm{d}_x\sigma^{-1}|=\frac{|\det\mathrm{d}_x\sigma^{-1}|}{|\det\mathrm{d}_{\widetilde{\gamma}(x)}\sigma^{-1}|}=\left(\frac{\psi(x)}{\psi(\widetilde{\gamma}(x))}\right)^{2^*}$$
and, since $\iota^{-1}(u)=\frac{1}{\psi\circ\sigma}u\circ\sigma$, we conclude that
\begin{equation}
\label{eq:gammatilde}
\gamma u = (\iota\circ\gamma\circ\iota^{-1})u = \frac{\psi}{\psi\circ\widetilde{\gamma}}u\circ\tilde{\gamma} = |\det\widetilde{\gamma}'|^{1/2^{*}}u\circ\widetilde{\gamma}.
\end{equation}

Using Lemma \ref{lem:invariance} it is easy to see that the functional $E$ is invariant under this action, i.e.,
\[
E(\gamma(u,v))=E(u,v)\text{\quad for every }\gamma\in O(N+1),\text{ }%
(u,v)\in\mathbf{D},
\]
and so are $f$ and $h.$ If $\Gamma$ is a closed subgroup of $O(N+1),$ we write%
\[
\mathbf{D}^{\Gamma}:=\{(u,v)\in\mathbf{D}:\gamma(u,v)=(u,v)\text{ \ for every
}\gamma\in\Gamma\}
\]
for the $\Gamma$-fixed point set of $\mathbf{D}.$ By \eqref{eq:gammatilde} we have that $(u,v)\in
\mathbf{D}^{\Gamma}$ iff $(u,v)$ is $\Gamma$-invariant in the sense defined in
the introduction. Define 
\[
\mathcal{N}^{\Gamma}:=\{(u,v)\in\mathbf{D}^{\Gamma}:u\neq0,\text{ }%
v\neq0,\text{ }f(u,v)=0,\text{ }h(u,v)=0\}.
\]
Recall that a group $\Gamma$ is said to act \textit{transitively} on a set $X$ if $X$ has only one $\Gamma$-orbit.

\begin{lemma}
\label{lem:nehari_nonempty}
If $\Gamma$ does not act transitively on $\mathbb{S}^{N}$, then $\mathcal{N}^{\Gamma}\neq\emptyset$.
\end{lemma}

\begin{proof}
Since $\Gamma$ does not act transitively on $\mathbb{S}^{N}$, there are two points in $\mathbb{S}^{N}$ whose $\Gamma$-orbits are disjoint. Taking two nontrivial $\Gamma$-invariant functions in $\mathcal{C}^{\infty}(\mathbb{S}^{N})$ whose supports lie in disjoint neighborhoods of these orbits, and composing them with the inverse of the stereographic projection, we obtain a pair of nontrivial functions $(u,v)\in\mathbf{D}^{\Gamma}$ with $\mathrm{supp}(u) \cap \mathrm{supp}(v) = \emptyset$. Setting $s,t\in(0,\infty)$ such that 
$$\int_{\mathbb{R}^{N}}|\nabla (su)|^{2}=\mu_{1}\int_{\mathbb{R}^{N}}|su|^{2^*} \qquad\text{and}\qquad\int_{\mathbb{R}^{N}}|\nabla (tv)|^{2}=\mu_{2}\int_{\mathbb{R}^{N}}|tv|^{2^*},$$
we get that $(su,tv)\in\mathcal{N}^{\Gamma}$.
\end{proof}

We assume from now on that $\Gamma$ does not act transitively on $\mathbb{S}^{N}$. 

It is easy to see that $\nabla E(u,v),\nabla f(u,v),\nabla h(u,v)\in\mathbf{D}^{\Gamma}$ for every $(u,v)\in\mathbf{D}^{\Gamma}$; cf.
Theorem 1.28 in \cite{w}. Then, it follows from Proposition \ref{prop:nehari} that $\mathcal{N}^{\Gamma}$ is a closed $\mathcal{C}^{1}$-submanifold of $\mathbf{D}^{\Gamma}$ and a natural constraint for $E.$ The tangent space to $\mathcal{N}^{\Gamma}$\ at the point $(u,v)$ is the orthogonal complement in $\mathbf{D}^{\Gamma}$ of the
linear subspace generated by $\nabla f(u,v)$ and $\nabla h(u,v).$

The following fact plays a crucial role in the proof of Proposition
\ref{prop:PS} below.

\begin{proposition}
\label{prop:compactness}If the $\Gamma$-orbit $\Gamma p:=\{\gamma p:\gamma\in\Gamma\}$\ of every point $p\in\mathbb{S}^{N}$ has positive dimension, then the embedding $\mathbf{D}^{\Gamma}\hookrightarrow L^{2^{\ast}}(\mathbb{R}^{N})\times L^{2^{\ast}}(\mathbb{R}^{N})$ is compact.
\end{proposition}

\begin{proof}
It is shown in \cite{d, hv} that the embedding $H_{g}^{1}(\mathbb{S}%
^{N})^{\Gamma}\hookrightarrow L_{g}^{2^{\ast}}(\mathbb{S}^{N})$ is compact.
The map defined in (\ref{eq:stereographic}) is an isometry between the
$\Gamma$-fixed point spaces $\iota:H_{g}^{1}(\mathbb{S}^{N})^{\Gamma
}\rightarrow D^{1,2}(\mathbb{R}^{N})^{\Gamma}$ and the Lebesgue spaces
$\iota:L_{g}^{2^{\ast}}(\mathbb{S}^{N})\rightarrow L^{2^{\ast}}(\mathbb{R}%
^{N}).$ Therefore, the embedding $D^{1,2}(\mathbb{R}^{N})^{\Gamma
}\hookrightarrow L^{2^{\ast}}(\mathbb{R}^{N})$ is compact, and so is%
\[
\mathbf{D}^{\Gamma}=D^{1,2}(\mathbb{R}^{N})^{\Gamma}\times D^{1,2}%
(\mathbb{R}^{N})^{\Gamma}\hookrightarrow L^{2^{\ast}}(\mathbb{R}^{N})\times
L^{2^{\ast}}(\mathbb{R}^{N}),
\]
as claimed.
\end{proof}

Let us give some examples.

\begin{example}
\label{example}

\begin{enumerate}
\item If $m+n=N+1,$ the group $\Gamma:=O(m)\times O(n)$ acts on $\mathbb{R}%
^{N+1}\equiv\mathbb{R}^{m}\times\mathbb{R}^{n}$ in the obvious way. The
$\Gamma$-orbit of a point $(x_{0},y_{0})\in\mathbb{R}^{m}\times\mathbb{R}^{n}$
is the set
\[
\Gamma(x_{0},y_{0})=\left\{  (x,y)\in\mathbb{R}^{m}\times\mathbb{R}%
^{n}:\left\vert x\right\vert =\left\vert x_{0}\right\vert ,\text{ }\left\vert
y\right\vert =\left\vert y_{0}\right\vert \right\}  ,
\]
so the $\Gamma$-orbit of every point $p\in\mathbb{S}^{N}$ has positive
dimension iff $m,n\geq2$.

\item For $N$ odd, another example is obtained by taking $\Gamma
:=\mathbb{S}^{1}$ to be the group of unit complex numbers acting on
$\mathbb{C}^{(N+1)/2}\equiv\mathbb{R}^{N+1}$ by multiplication on each complex
coordinate. Then, the $\Gamma$-orbit of every point in $\mathbb{S}^{N}$ is a circle.
\end{enumerate}
\end{example}

We write $\nabla_{\mathcal{N}}E(u,v)$ for the orthogonal projection of $\nabla
E(u,v)$ onto the tangent space of $\mathcal{N}$ at $(u,v).$

\begin{lemma}
\label{lem:PS}If $((u_{k},v_{k}))$ is a sequence in $\mathcal{N}$ such that%
\[
E(u_{k},v_{k})\rightarrow c\text{\qquad and\qquad}\nabla_{\mathcal{N}}%
E(u_{k},v_{k})\rightarrow0,
\]
then $((u_{k},v_{k}))$ is bounded in $\mathbf{D}$ and $\nabla E(u_{k}%
,v_{k})\rightarrow0.$
\end{lemma}

\begin{proof}
If $(u_{k},v_{k})\in\mathcal{N}$, $E(u_{k},v_{k})\to c$ and $\nabla_{\mathcal{N}}E(u_{k},v_{k})\to 0$ then, as
\[
E(u_{k},v_{k})=E(u_{k},v_{k})-\frac{1}{2^{\ast}}E^{\prime}(u_{k},v_{k})\left[
(u_{k},v_{k})\right]  =\frac{1}{N}\int_{\mathbb{R}^{N}}\left(  \left\vert
\nabla u_{k}\right\vert ^{2}+\left\vert \nabla v_{k}\right\vert ^{2}\right)
,
\]
we have that $((u_{k},v_{k}))$ is bounded in $\mathbf{D.}$ This easily implies
that $(\nabla f(u_{k},v_{k}))$ and $(\nabla h(u_{k},v_{k}))$ are bounded in
$\mathbf{D.}$ Let $s_{k},t_{k}\in\mathbb{R}$ be such that
\begin{equation}
\nabla E(u_{k},v_{k})=\nabla_{\mathcal{N}}E(u_{k},v_{k})+s_{k}\nabla
f(u_{k},v_{k})+t_{k}\nabla h(u_{k},v_{k}). \label{eq:gradient}%
\end{equation}
As $(u_{k},v_{k})\in\mathcal{N}$ and $\nabla_{\mathcal{N}}E(u_{k}%
,v_{k})\rightarrow0,$ taking the scalar product of this identity with
$(u_{k},0)$ and $(0,v_{k}),$ we get that $s_{k}$ and $t_{k}$ solve the system%
\begin{equation}
\left\{
\begin{array}
[c]{c}%
o(1)=s_{k}a_{11}^{(k)}+t_{k}a_{12}^{(k)},\\
o(1)=s_{k}a_{12}^{(k)}+t_{k}a_{22}^{(k)},
\end{array}
\right.  \label{eq:system_k}%
\end{equation}
where $o(1)\rightarrow0$ as $k\rightarrow\infty,$%
\begin{align*}
a_{11}^{(k)}  &  :=\left(  2-2^{\ast}\right)  \mu_{1}\int_{\mathbb{R}^{N}%
}|u_{k}|^{{2}^{\ast}}+\lambda\alpha(2-\alpha)\int_{\mathbb{R}^{N}}%
|u_{k}|^{{\alpha}}|v_{k}|^{{\beta}},\\
a_{12}^{(k)}  &  :=-\lambda\alpha\beta\int_{\mathbb{R}^{N}}|u_{k}|^{{\alpha}%
}|v_{k}|^{{\beta}}=:a_{21}^{(k)}\\
a_{22}^{(k)}  &  :=\left(  2-2^{\ast}\right)  \mu_{2}\int_{\mathbb{R}^{N}%
}|v_{k}|^{{2}^{\ast}}+\lambda\beta(2-\beta)\int_{\mathbb{R}^{N}}%
|u_{k}|^{{\alpha}}|v_{k}|^{{\beta}}.
\end{align*}
After passing to a subsequence, we have that $\int_{\mathbb{R}^{N}}%
|u_{k}|^{{\alpha}}|v_{k}|^{{\beta}}\rightarrow b\in\lbrack0,\infty).$ If
$b=0,$ the statement (a) of Proposition \ref{prop:nehari}\ implies that
\[
\det(a_{ij}^{(k)})\geq\frac{1}{2}\left(  2-2^{\ast}\right)  ^{2}c_{0}%
^{2}>0\text{\qquad for }k\text{ large enough,}%
\]
where $c_{0}:=\min\{\mu_{1}^{-(N-2)/2},\mu_{2}^{-(N-2)/2}\}S^{N/2}.$ If $b>0,$
then (\ref{eq:det}) implies that
\begin{align*}
\det(a_{ij}^{(k)})  &  \geq\left(  2^{\ast}-2\right)  c_{0}\alpha\beta\left(
-\lambda\right)  \int_{\mathbb{R}^{N}}|u_{k}|^{{\alpha}}|v_{k}|^{{\beta}}\\
&  \geq\frac{1}{2}\left(  2^{\ast}-2\right)  c_{0}\alpha\beta\left(
-\lambda\right)  d>0\text{\qquad for }k\text{ large enough.}%
\end{align*}
Therefore, the system (\ref{eq:system_k}) has a unique solution $(s_{k}%
,t_{k})$ for large enough $k$ and, as $((u_{k},v_{k}))$ is bounded in
$\mathbf{D,}$ after passing to a subsequence, we conclude that $s_{k}%
\rightarrow0$ and $t_{k}\rightarrow0.$ From the identity (\ref{eq:gradient})
we get that $\nabla E(u_{k},v_{k})\rightarrow0,$ as claimed.
\end{proof}

\begin{proposition}
\label{prop:PS}If the $\Gamma$-orbit of every point $p\in\mathbb{S}^{N}$ has
positive dimension, then every sequence $((u_{k},v_{k}))$ in $\mathcal{N}%
^{\Gamma}$ such that%
\[
E(u_{k},v_{k})\rightarrow c\text{\qquad and\qquad}\nabla_{\mathcal{N}}%
E(u_{k},v_{k})\rightarrow0
\]
contains a convergent subsequence.
\end{proposition}

\begin{proof}
By Lemma \ref{lem:PS}\ and Proposition \ref{prop:compactness}, $((u_{k}%
,v_{k}))$ is bounded in $\mathbf{D}$ and the embedding $\mathbf{D}^{\Gamma
}\hookrightarrow L^{2^{\ast}}(\mathbb{R}^{N})\times L^{2^{\ast}}%
(\mathbb{R}^{N})$ is compact. So, after passing to a subsequence, we have that
$(u_{k},v_{k})\rightharpoonup(u,v)$ weakly in $\mathbf{D}$ and $(u_{k}%
,v_{k})\rightarrow(u,v)$ strongly in $L^{2^{\ast}}(\mathbb{R}^{N})\times
L^{2^{\ast}}(\mathbb{R}^{N}).$ It follows that%
\begin{align*}
&  \left\vert \mu_{1}\int_{\mathbb{R}^{N}}|u_{k}|^{{2}^{\ast}-2}u_{k}%
(u_{k}-u)+\lambda\alpha\int_{\mathbb{R}^{N}}|u_{k}|^{{\alpha-2}}u_{k}%
(u_{k}-u)|v_{k}|^{{\beta}}\right\vert \\
&  \leq\mu_{1}|u_{k}|_{2^{\ast}}^{{2}^{\ast}-1}|u_{k}-u|_{2^{\ast}}%
-\lambda\alpha|u_{k}|_{2^{\ast}}^{\alpha-1}|v_{k}|_{2^{\ast}}^{\beta}%
|u_{k}-u|_{2^{\ast}}\\
&  \leq C|u_{k}-u|_{2^{\ast}}=o(1),
\end{align*}
where $|\cdot|_{2^{\ast}}$ denotes the norm in $L^{2^{\ast}}(\mathbb{R}^{N}).$
As $\nabla E(u_{k},v_{k})\rightarrow0,$ we get that%
\begin{align*}
o(1)  &  =\partial_{u}E(u_{k},v_{k})\left[  u_{k}-u\right] \\
&  =\int_{\mathbb{R}^{N}}\nabla u_{k}\cdot\nabla(u_{k}-u)\\
&  \qquad-\mu_{1}\int_{\mathbb{R}^{N}}|u_{k}|^{{2}^{\ast}-2}u_{k}%
(u_{k}-u)-\lambda\alpha\int_{\mathbb{R}^{N}}|u_{k}|^{{\alpha-2}}u_{k}%
(u_{k}-u)|v_{k}|^{{\beta}}\\
&  =\int_{\mathbb{R}^{N}}\left\vert \nabla u_{k}\right\vert ^{2}%
-\int_{\mathbb{R}^{N}}\left\vert \nabla u\right\vert ^{2}+o(1).
\end{align*}
Therefore, $u_{k}\rightarrow u$ strongly in $D^{1,2}(\mathbb{R}^{N}).$
Similarly, as $\partial_{v}E(u_{k},v_{k})\left[  v_{k}-v\right]  =o(1),$ we
get that $v_{k}\rightarrow v$ strongly in $D^{1,2}(\mathbb{R}^{N}).$
\end{proof}

\begin{theorem}
\label{thm:existence}If $\Gamma$ does not act transitively on $\mathbb{S}^{N}$ and  the $\Gamma$-orbit of every point $p\in\mathbb{S}^{N}$ has positive dimension, then $E$ has a positive minimizer on
$\mathcal{N}^{\Gamma}.$
\end{theorem}

\begin{proof}
We have shown that $\mathcal{N}^{\Gamma}$ is a $\mathcal{C}^{1}$-submanifold
of $\mathbf{D}^{\Gamma}$ and that $E$ is of class $\mathcal{C}^{1},$
bounded below and satisfies the Palais-Smale condition on $\mathcal{N}%
^{\Gamma};$ see Propositions \ref{prop:nehari}, \ref{prop:nonexistence}\ and
\ref{prop:PS}. Since $\mathcal{N}^{\Gamma}\neq\emptyset,$ Theorem 3.1 in
\cite{sz} asserts that $c_{1}=\inf_{(u,v)\in\mathcal{N}^{\Gamma}}E(u,v)$ is
attained. As $E(\left\vert u\right\vert ,\left\vert v\right\vert )=E(u,v),$
$E$ has a positive minimizer on $\mathcal{N}^{\Gamma}.$
\end{proof}

\section{Multiplicity for the symmetric system}

\label{sec:multiplicity}To obtain multiple solutions, we adapt a
$\mathcal{C}^{1}$-Ljusternik-Schnirelmann result, which was proved by A.
Szulkin in \cite{sz}.

Let $X$ be a real Banach space with an action of the group $\mathbb{Z}%
_{2}:=\{1,-1\}$ by linear isometries. A point $z\in X$ is called a fixed point
if $(-1)\cdot z=z.$ A $\mathbb{Z}_{2}$-invariant subset of $X$ is called
\emph{fixed point free} if it does not contain a fixed point.

Let $\Sigma$ be the collection of all closed $\mathbb{Z}_{2}$-invariant
subsets of $X$ which are fixed point free. If $Z\in\Sigma$ is nonempty$,$ the
\emph{genus of} $Z$ is the smallest integer $j\geq1$ such that there exists a
continuous function $\phi:Z\rightarrow\mathbb{S}^{j-1}$ which is
$\mathbb{Z}_{2}$-equivariant, i.e., $\phi((-1)\cdot z)=-\phi(z)$ for all $z\in
Z.$ We denote this integer by genus$(Z).$ If no such $j$ exists we set
genus$(Z):=\infty.$ We define genus$(\emptyset):=0.$

\begin{theorem}
\label{thm:szulkin}Let $M$ be a closed $\mathbb{Z}_{2}$-invariant
$\mathcal{C}^{1}$-submanifold of $X$ which is fixed point free, and let
$F\in\mathcal{C}^{1}(M,\mathbb{R})$ be a $\mathbb{Z}_{2}$-invariant function
which is bounded below and satisfies the Palais-Smale condition. If the set%
\[
\Sigma_{j}:=\{Z\in\Sigma:Z\subset M,\text{ }Z\text{ is compact and
\emph{genus}}(Z)\geq j\}
\]
is nonempty for every $j\geq1$, then $F$ has infinitely many critical values.
\end{theorem}

\begin{proof}
Let $K_{c}:=\{z\in M:F(z)=c$ and $F^{\prime}(z)=0\},$ and set%
\[
c_{j}:=\inf_{Z\in\Sigma_{j}}\max_{z\in Z}F(z).
\]
Since $F$ is bounded below, $\Sigma_{j+1}\subset\Sigma_{j}$ and the sets $Z$
in $\Sigma_{j}$ are compact, we have that%
\[
-\infty<c_{1}\leq c_{2}\leq\cdots\leq c_{j}\leq\cdots<\infty.
\]
The proof of Theorem 3.1 in \cite{sz} can be adapted, in a straightforward
manner, to show that, if $c_{j}=\cdots=c_{j+m}=:c$ for some $m\geq0,$ then%
\[
\text{genus}(K_{c})\geq m+1.
\]
One needs only to replace cat$_{M}$ by genus, and take care that the sets
involved are $\mathbb{Z}_{2}$-invariant and the maps are $\mathbb{Z}_{2}$-equivariant.

In particular, genus$(K_{c_{j}})\geq1.$ Hence, $c_{j}$ is a critical value.
Moreover, as $F$ satisfies the Palais-Smale condition, the sets $K_{c}$ are
compact and have, therefore, finite genus. It follows that, for each $j\geq1$
there exists $m>0$ such that $c_{j}\neq c_{j+m}.$ This proves our claim.
\end{proof}

We derive the following result.

\begin{theorem}
\label{thm:multiplicity}Assume that $\mu_{1}=\mu_{2}=:\mu$ and $\alpha=\beta.$
If $\Gamma$ does not act transitively on $\mathbb{S}^{N}$ and the $\Gamma$-orbit of every point $p\in\mathbb{S}^{N}$ has positive
dimension, then $E:\mathcal{N}^{\Gamma}\rightarrow\mathbb{R}$ has infinitely
many critical values for each $\lambda\leq-\frac{\mu}{\alpha}.$
\end{theorem}

\begin{proof}
Consider the action of the group $\mathbb{Z}_{2}$ on $\mathbf{D}$\ given by%
\[
(-1)\cdot(u,v):=(-v,-u).
\]
As $\mu_{1}=\mu_{2}=:\mu$ and $\alpha=\beta$, we have that $(-v,-u)\in
\mathcal{N}^{\Gamma}$\ iff \ $(u,v)\in\mathcal{N}^{\Gamma}$%
,\ and\ $E(-v,-u)=E(u,v)$ for all $(u,v)\in\mathcal{N}^{\Gamma}.$ Note also
that $(u,-u)\notin\mathcal{N}$ if $\lambda\leq-\frac{\mu}{\alpha}.$
Otherwise,
\[
\int_{\mathbb{R}^{N}}\left\vert \nabla u\right\vert ^{2}=(\mu+\lambda
\alpha)\int_{\mathbb{R}^{N}}|u|^{{2}^{\ast}}\leq0,
\]
which is a contradiction. This means that $E$ and $\mathcal{N}^{\Gamma}$ are
$\mathbb{Z}_{2}$-invariant, and $\mathcal{N}^{\Gamma}$ is fixed point free. We
have already shown that $E$ is bounded below and satisfies the Palais-Smale
condition on $\mathcal{N}^{\Gamma};$ see Propositions \ref{prop:nonexistence}%
\ and \ref{prop:PS}. So, all that is left, is to show is that $\mathcal{N}%
^{\Gamma}$ contains a compact $\mathbb{Z}_{2}$-invariant subset of genus $\geq
j,$ for each $j\geq1.$

Fix $j\geq1$. As $\Gamma$ does not act transitively on $\mathbb{S}^N$, we may choose $2j$ pairwise disjoint open $\Gamma$-invariant subsets $U_1,...,U_{2j}$ of $\mathbb{R}^N$ and nontrivial $\Gamma$-invariant functions $u_i\in\mathcal{C}_c^{\infty}(U_i)$, $v_i\in\mathcal{C}_c^{\infty}(U_{j+i})$, for each $i=1,...,j$. 
For $w\in D^{1,2}(\mathbb{R}^N)$, $w\neq 0$, let $t_w$ be the unique positive number such that
$\int_{\mathbb{R}^{N}}|\nabla (t_ww)|^2=\mu\int_{\mathbb{R}^{N}}|t_ww|^{2^*}$ and, for $(u,v)\in\mathbf{D}^{\Gamma}$ with $u\neq 0$ and $v\neq 0$, define
$$\varrho(u,v):= (t_uu,\,t_vv).$$
Note that $\varrho[(-1)\cdot(u,v)]= (-1)\cdot\varrho(u,v)$ and that $\varrho(u,v)\in\mathcal{N}^{\Gamma}$ if $uv=0$. Let $\{e_1,...,e_j\}$ be the canonical basis of $\mathbb{R}^j$. The boundary of the convex hull of the set $\{\pm e_1,...,\pm e_j\}$, which is given by
$$Q := \left\{\sum\limits_{i=1}^j \lambda_i\tilde{e}_i:\,\tilde{e}_i\in\{e_i,-e_i\},\,\lambda_i\in[0,1],\,\sum\limits_{i=1}^j\lambda_i = 1 \right\},$$
is symmetric with respect to the origin and radially homeomorphic to the unit sphere $\mathbb{S}^{j-1}$. Setting $h(e_i):= \varrho(u_i,v_i)$ and $h(-e_i):= \varrho(-v_i,-u_i)$, and extending this map by
$$h\left(\sum\limits_{i=1}^j \lambda_i\tilde{e}_i\right):=\varrho\left(\sum\limits_{i=1}^j \lambda_ih(\tilde{e}_i)\right),$$
we obtain a map $h:Q \to \mathcal{N}^{\Gamma}$ which is well defined, continuous and $\mathbb{Z}_{2}$-equivariant. Then, the set $Z:=h(Q)$ is compact and $\mathbb{Z}_{2}$-invariant. If $\phi:Z\to\mathbb{S}^{k-1}$ is continuous and $\mathbb{Z}_{2}$-equivariant, the composition $\phi\circ h$ yields an odd map $\mathbb{S}^{j-1}\to\mathbb{S}^{k-1}.$ The Borsuk-Ulam theorem implies that $k\geq j$. Hence, $\mathrm{genus}(Z)\geq j$ This finishes the proof.
\end{proof}

\smallskip

\begin{proof}
[Proof of Theorem \ref{thm:main}]The statements (a) and (b) follow from
Theorems \ref{thm:existence} and \ref{thm:multiplicity} respectively. The
statement (c) follows from Proposition \ref{prop:nonsyncronized}.
\end{proof}

\section{Phase separation}

\label{sec:separation}This section is devoted to the proof of Theorem
\ref{thm:separation}. We assume throughout that $\Gamma\ $is a closed subgroup
of $O(N+1)$ which does not act transitively on $\mathbb{S}^{N}$ and that the $\Gamma$-orbit of every point $p\in\mathbb{S}^{N}$ has positive dimension. To highlight the dependence on $\lambda,$ we write
$E_{\lambda},\mathcal{N}_{\lambda}^{\Gamma},f_{\lambda},h_{\lambda}$ instead
of $E,\mathcal{N}^{\Gamma},f,h,$ and we set%
\[
c_{\lambda}^{\Gamma}:=\inf_{(u,v)\in\mathcal{N}_{\lambda}^{\Gamma}}E_{\lambda
}(u,v).
\]
We denote by $J$ and $\mathcal{M}^{\Gamma}$ the energy functional and the
Nehari manifold of the problem%
\begin{equation}
-\Delta w=\mu_{1}|w^{+}|^{{2}^{\ast}-2}w^{+}+\mu_{2}|w^{-}|^{{2}^{\ast}%
-2}w^{-},\text{\qquad}w\in D^{1,2}(\mathbb{R}^{N})^{\Gamma}, \label{limitprob}%
\end{equation}
where $w^{+}:=\max\{w,0\}$ and $w^{-}:=\min\{w,0\},$ i.e.,%
\[
J(w):=\frac{1}{2}\int_{\mathbb{R}^{N}}\left\vert \nabla w\right\vert
^{2}-\frac{1}{2^{\ast}}\int_{\mathbb{R}^{N}}(\mu_{1}|w^{+}|^{2^{\ast}}+\mu
_{2}|w^{-}|^{2^{\ast}}),
\]
and
\begin{align*}
\mathcal{M}^{\Gamma}:  &  =\{w\in D^{1,2}(\mathbb{R}^{N})^{\Gamma}%
:w\neq0,\text{ }J^{\prime}(w)w=0\}\\
&  =\left\{  w\in D^{1,2}(\mathbb{R}^{N})^{\Gamma}:w\neq0,\text{ }%
\int_{\mathbb{R}^{N}}\left\vert \nabla w\right\vert ^{2}=\int_{\mathbb{R}^{N}%
}(\mu_{1}|w^{+}|^{2^{\ast}}+\mu_{2}|w^{-}|^{2^{\ast}})\right\}  .
\end{align*}
The sign-changing solutions of (\ref{limitprob}) lie on the set%
\[
\mathcal{E}^{\Gamma}:=\{w\in D^{1,2}(\mathbb{R}^{N})^{\Gamma}:w^{+}%
\in\mathcal{M}^{\Gamma},\text{ }w^{-}\in\mathcal{M}^{\Gamma}\}.
\]
Note that, if $u,v\in D^{1,2}(\mathbb{R}^{N})^{\Gamma}\smallsetminus\{0\},$
$u\geq0$ and $v\geq0,$ then there exist unique numbers $s,t\in(0,\infty)$ such
that $su\in\mathcal{M}^{\Gamma}$ and$\,-tv\in\mathcal{M}^{\Gamma},$ namely,
\begin{equation}
s^{2^{\ast}-2}=\frac{\int_{\mathbb{R}^{N}}\left\vert \nabla u\right\vert ^{2}%
}{\int_{\mathbb{R}^{N}}\mu_{1}|u|^{2^{\ast}}}\text{\qquad and\qquad}%
t^{2^{\ast}-2}=\frac{\int_{\mathbb{R}^{N}}\left\vert \nabla v\right\vert ^{2}%
}{\int_{\mathbb{R}^{N}}\mu_{2}|v|^{2^{\ast}}}. \label{projection}%
\end{equation}
If, moreover, $uv=0,$ then $su-tv\in\mathcal{E}^{\Gamma}.$ Arguing as in the proof of Lemma \ref{lem:nehari_nonempty}, we see that there exist $u$ and $v$ with these properties. Hence, $\mathcal{E}^{\Gamma} \neq \emptyset$. We define
\[
c_{\infty}^{\Gamma}:=\inf_{w\in\mathcal{E}^{\Gamma}}J(w)<\infty.
\]

\begin{proposition}
\label{prop:separation}For $\lambda_{k}\rightarrow-\infty,$ let $(u_{k}%
,v_{k})\in\mathcal{N}_{\lambda_{k}}^{\Gamma}$ satisfy $u_{k}\geq0,$ $v_{k}%
\geq0$ and $E_{\lambda_{k}}(u_{k},v_{k})=c_{\lambda_{k}}^{\Gamma}.$ Then,
after passing to a subsequence, we have that $u_{k}\rightarrow u_{\infty}$ and
$v_{k}\rightarrow v_{\infty}$ strongly in $D^{1,2}(\mathbb{R}^{N})^{\Gamma},$
and these functions satisfy

\begin{enumerate}
\item[(a)] $u_{\infty},v_{\infty}\in\mathcal{M}^{\Gamma},$ $u_{\infty}\geq0,$
$v_{\infty}\geq0,$ $u_{\infty}v_{\infty}=0.$ Thus, $u_{\infty}-v_{\infty}%
\in\mathcal{E}^{\Gamma}.$

\item[(b)] $\lim_{k\rightarrow\infty}c_{\lambda_{k}}^{\Gamma}=J(u_{\infty
}-v_{\infty})=c_{\infty}^{\Gamma}.$

\item[(c)] $u_{\infty}-v_{\infty}$ solves the problem \emph{(\ref{limitprob})}.
\end{enumerate}
\end{proposition}

\begin{proof}
If $w\in\mathcal{E}^{\Gamma}$ then, as $w^{+}w^{-}=0,$ we have that
$(w^{+},w^{-})\in\mathcal{N}_{\lambda}^{\Gamma}$ and $J(w)=E_{\lambda}(w^{+}%
,w^{-})$ for every $\lambda<0.$ Hence,
\[
c_{\lambda}^{\Gamma}\leq c_{\infty}^{\Gamma}\text{\qquad for every \ }%
\lambda<0.
\]
This implies, in particular, that
\[
\frac{1}{N}\int_{\mathbb{R}^{N}}\left(  \left\vert \nabla u_{k}\right\vert
^{2}+\left\vert \nabla v_{k}\right\vert ^{2}\right)  =E_{\lambda_{k}}%
(u_{k},v_{k})\leq c_{\infty}^{\Gamma}\text{\qquad for all }k\in\mathbb{N}.
\]
So, after passing to a subsequence, there exist $u_{\infty},v_{\infty}\in
D^{1,2}(\mathbb{R}^{N})^{\Gamma}$ such that
\begin{align*}
u_{k}  &  \rightharpoonup u_{\infty},\text{\qquad}v_{k}\rightharpoonup
v_{\infty},\text{\qquad weakly in }D^{1,2}(\mathbb{R}^{N}),\\
u_{k}  &  \rightarrow u_{\infty},\text{\qquad}v_{k}\rightarrow v_{\infty
},\text{\qquad strongly in }L^{2^{\ast}}(\mathbb{R}^{N}),\\
u_{k}  &  \rightarrow u_{\infty},\text{\qquad}v_{k}\rightarrow v_{\infty
},\text{\qquad a.e. in }\mathbb{R}^{N}.
\end{align*}
Hence, $u_{\infty}\geq0$ and $v_{\infty}\geq0.$ Since $f_{\lambda_{k}}%
(u_{k},v_{k})+h_{\lambda_{k}}(u_{k},v_{k})=0,$ we have that%
\[
0\leq2^{\ast}(-\lambda_{k})\int_{\mathbb{R}^{N}}\left\vert u_{k}\right\vert
^{\alpha}\left\vert v_{k}\right\vert ^{\beta}\leq\mu_{1}\int_{\mathbb{R}^{N}%
}|u_{k}|^{2^{\ast}}+\mu_{2}\int_{\mathbb{R}^{N}}|v_{k}|^{2^{\ast}}\leq C_{0}%
\]
and, using Fatou's lemma, we obtain%
\[
\int_{\mathbb{R}^{N}}\left\vert u_{\infty}\right\vert ^{\alpha}\left\vert
v_{\infty}\right\vert ^{\beta}\leq\liminf_{k\rightarrow\infty}\int
_{\mathbb{R}^{N}}\left\vert u_{k}\right\vert ^{\alpha}\left\vert
v_{k}\right\vert ^{\beta}\leq\frac{C_{0}}{2^{\ast}}\lim_{k\rightarrow\infty
}\frac{1}{(-\lambda_{k})}=0.
\]
Thus, $u_{\infty}v_{\infty}=0.$ On the other hand, Proposition
\ref{prop:nehari}(a) yields%
\begin{align*}
0  &  <c_{0}\leq\int_{\mathbb{R}^{N}}\left\vert \nabla u_{k}\right\vert
^{2}\leq\mu_{1}\int_{\mathbb{R}^{N}}|u_{k}|^{2^{\ast}},\\
0  &  <c_{0}\leq\int_{\mathbb{R}^{N}}\left\vert \nabla v_{k}\right\vert
^{2}\leq\mu_{2}\int_{\mathbb{R}^{N}}|v_{k}|^{2^{\ast}},
\end{align*}
Therefore, $u_{\infty}\neq0$ and $v_{\infty}\neq0.$ Then, as in
(\ref{projection}), there exist $s,t\in(0,\infty)$ such that $su_{\infty
},\,-tv_{\infty}\in\mathcal{M}^{\Gamma}$ and $su_{\infty}-tv_{\infty}%
\in\mathcal{E}^{\Gamma}$. So, after passing to a subsequence, we obtain%
\begin{align*}
c_{\infty}^{\Gamma}  &  \leq\frac{1}{2}\int_{\mathbb{R}^{N}}(\left\vert
\nabla\left(  su_{\infty}\right)  \right\vert ^{2}+\left\vert \nabla\left(
tv_{\infty}\right)  \right\vert ^{2})-\frac{1}{2^{\ast}}\int_{\mathbb{R}^{N}%
}(\mu_{1}|su_{\infty}|^{{2}^{\ast}}+\mu_{2}|tv_{\infty}|^{{2}^{\ast}})\\
&  \leq\frac{1}{2}\liminf_{k\rightarrow\infty}\int_{\mathbb{R}^{N}}(\left\vert
\nabla\left(  su_{k}\right)  \right\vert ^{2}+\left\vert \nabla\left(
tv_{k}\right)  \right\vert ^{2})-\frac{1}{2^{\ast}}\lim_{k\rightarrow\infty
}\int_{\mathbb{R}^{N}}(\mu_{1}|su_{k}|^{{2}^{\ast}}+\mu_{2}|tv_{k}|^{{2}%
^{\ast}})\\
&  \leq\frac{1}{2}\liminf_{k\rightarrow\infty}\int_{\mathbb{R}^{N}}(\left\vert
\nabla\left(  su_{k}\right)  \right\vert ^{2}+\left\vert \nabla\left(
tv_{k}\right)  \right\vert ^{2})-\frac{1}{2^{\ast}}\lim_{k\rightarrow\infty
}\int_{\mathbb{R}^{N}}(\mu_{1}|su_{k}|^{{2}^{\ast}}+\mu_{2}|tv_{k}|^{{2}%
^{\ast}})\\
&  \qquad+\lim_{k\rightarrow\infty}(-\lambda_{k})\int_{\mathbb{R}^{N}%
}\left\vert su_{k}\right\vert ^{\alpha}\left\vert tv_{k}\right\vert ^{\beta}\\
&  \leq\liminf_{k\rightarrow\infty}E_{\lambda_{k}}(su_{k},tv_{k})\leq
\liminf_{k\rightarrow\infty}E_{\lambda_{k}}(u_{k},v_{k})=\liminf
_{k\rightarrow\infty}c_{\lambda_{k}}^{\Gamma}\leq\limsup_{k\rightarrow\infty
}c_{\lambda_{k}}^{\Gamma}\leq c_{\infty}^{\Gamma},
\end{align*}
because $E_{\lambda_{k}}(su_{k},tv_{k})\leq E_{\lambda_{k}}(u_{k},v_{k});$ see
Proposition \ref{prop:nehari}(d). It follows that
\[
\lim_{k\rightarrow\infty}(-\lambda_{k})\int_{\mathbb{R}^{N}}\left\vert
u_{k}\right\vert ^{\alpha}\left\vert v_{k}\right\vert ^{\beta}=0
\]
and that
\begin{align*}
c_{\infty}^{\Gamma}  &  =\lim_{k\rightarrow\infty}c_{\lambda_{k}}^{\Gamma
}=\lim_{k\rightarrow\infty}E_{\lambda_{k}}(u_{k},v_{k})=\liminf_{k\rightarrow
\infty}E_{\lambda_{k}}(su_{k},tv_{k})\\
&  \leq\limsup_{k\rightarrow\infty}E_{\lambda_{k}}(su_{k},tv_{k})\leq
\lim_{k\rightarrow\infty}E_{\lambda_{k}}(u_{k},v_{k})=c_{\infty}^{\Gamma}.
\end{align*}
Hence,%
\[
\lim_{k\rightarrow\infty}\int_{\mathbb{R}^{N}}(\left\vert \nabla\left(
su_{k}\right)  \right\vert ^{2}+\left\vert \nabla\left(  tv_{k}\right)
\right\vert ^{2})=\int_{\mathbb{R}^{N}}(\left\vert \nabla\left(  su_{\infty
}\right)  \right\vert ^{2}+\left\vert \nabla\left(  tv_{\infty}\right)
\right\vert ^{2})
\]
and, as $su_{k}\rightharpoonup su_{\infty}$ and\ $tv_{k}\rightharpoonup
tv_{\infty}$ weakly in $D^{1,2}(\mathbb{R}^{N}),$ we conclude that
$u_{k}\rightarrow u_{\infty}$ and $v_{k}\rightarrow v_{\infty}$ strongly in
$D^{1,2}(\mathbb{R}^{N}).$ Consequently,%
\begin{align*}
c_{\infty}^{\Gamma}  &  =\lim_{k\rightarrow\infty}E_{\lambda_{k}}(u_{k}%
,v_{k})\\
&  =\frac{1}{2}\int_{\mathbb{R}^{N}}(\left\vert \nabla u_{\infty}\right\vert
^{2}+\left\vert \nabla v_{\infty}\right\vert ^{2})-\frac{1}{2^{\ast}}%
\int_{\mathbb{R}^{N}}(\mu_{1}|u_{\infty}|^{{2}^{\ast}}+\mu_{2}|v_{\infty
}|^{{2}^{\ast}})=J(u_{\infty}-v_{\infty}),
\end{align*}
and%
\begin{align*}
0  &  =\lim_{k\rightarrow\infty}f_{\lambda_{k}}(u_{k},v_{k})=\int
_{\mathbb{R}^{N}}\left\vert \nabla u_{\infty}\right\vert ^{2}-\mu_{1}%
\int_{\mathbb{R}^{N}}|u_{\infty}|^{{2}^{\ast}},\\
0  &  =\lim_{k\rightarrow\infty}h_{\lambda_{k}}(u_{k},v_{k})=\int
_{\mathbb{R}^{N}}\left\vert \nabla v_{\infty}\right\vert ^{2}-\mu_{2}%
\int_{\mathbb{R}^{N}}|v_{\infty}|^{{2}^{\ast}}.
\end{align*}
This shows that $u_{\infty},v_{\infty}\in\mathcal{M}^{\Gamma},$ and completes
the proof of (a) and (b).

Thus, we have shown that $u_{\infty}-v_{\infty}$ is a minimizer for $J$ on
$\mathcal{E}^{\Gamma}.$ Since the embedding $D^{1,2}(\mathbb{R}^{N})^{\Gamma
}\hookrightarrow L^{2^{\ast}}(\mathbb{R}^{N})$ is compact, $J$ satisfies the
Palais-Smale condition on $\mathcal{M}^{\Gamma}.$ So the same argument given
in \cite{ccn} to prove Lemma 2.6 leads to the conclusion that $u_{\infty
}-v_{\infty}$ is a critical point of $J.$ This proves (c).
\end{proof}

\smallskip

\begin{proof}
[Proof of Theorem \ref{thm:separation}]For $\lambda_{k}\rightarrow-\infty,$
let $(u_{k},v_{k})\in\mathcal{N}_{\lambda_{k}}^{\Gamma}$ satisfy $u_{k}\geq0,$
$v_{k}\geq0$ and $E_{\lambda_{k}}(u_{k},v_{k})=c_{\lambda_{k}}^{\Gamma}.$ By
Proposition \ref{prop:separation}, after passing to a subsequence, we have
that $u_{k}\rightarrow u_{\infty}$ and $v_{k}\rightarrow v_{\infty}$ strongly
in $D^{1,2}(\mathbb{R}^{N}),$ $u_{\infty}\geq0,$ $v_{\infty}\geq0,$
$u_{\infty}v_{\infty}=0,$ and $u_{\infty}-v_{\infty}$ is a nontrivial solution
to the problem (\ref{limitprob}). Then, a well known regularity argument shows
that $u_{\infty}-v_{\infty}\in\mathcal{C}^{1}(\mathbb{R}^{N});$ see, e.g.,
Appendix B in \cite{s}. As $u_{\infty}=(u_{\infty}-v_{\infty})^{+}$ and
$-v_{\infty}=(u_{\infty}-v_{\infty})^{-},$ these functions are continuous and
the sets $\Omega_{1}:=\{x\in\mathbb{R}^{N}:u_{\infty}(x)>0\}$ and $\Omega
_{2}:=\{x\in\mathbb{R}^{N}:v_{\infty}(x)>0\}$ are open. Since $u_{\infty}$ and $v_{\infty}$ are $\Gamma$-invariant, $\Omega_{1}$ and $\Omega_{2}$ are $\Gamma$-invariant and, as $u_{\infty}-v_{\infty}$ is a minimizer of $J$ on $\mathcal{E}^{\Gamma}$, these sets are connected. Moreover, we have that $\overline{\Omega_{1}\cup\Omega_{2}}=\mathbb{R}^{N}$ because, otherwise, $u_{\infty}-v_{\infty}$ would vanish in an open set, contradicting the unique continuation principle. Clearly, $u_{\infty}$ solves the problem%
\[
-\Delta u=\mu_{1}|u|^{{2}^{\ast}-2}u,\text{\qquad}u\in D_{0}^{1,2}(\Omega_{1}),
\]
and $v_{\infty}$ solves the problem%
\[
-\Delta v=\mu_{2}|v|^{{2}^{\ast}-2}v,\qquad v\in D_{0}^{1,2}(\Omega_{2}).
\]
This finishes the proof.
\end{proof}

Next, we prove Proposition \ref{prop:tori}. The proof is based on the following geometric lemma. We write $\mathbb{B}^k$ and $\mathbb{S}^{k-1}$ for the open ball and the unit sphere in $\mathbb{R}^k$.

\begin{lemma}
\label{lem:tori}Let $\Gamma=O(m)\times O(n)$ with $m+n=N+1$, $m,n\geq 2$, and let $U_1$ and $U_2$ be nonempty $\Gamma$-invariant open connected subsets of $\mathbb{S}^N$ such that $\overline{U_1\cup U_2}=\mathbb{S}^N$ and $U_1\cap U_2=\emptyset$. Then, up to relabeling, 
\begin{itemize}
\item[(a)]$\widetilde{U}_1:=U_1\cup (\mathbb{S}^{m-1}\times\{0\})$ and $\widetilde{U}_2:=U_2\cup (\{0\}\times\mathbb{S}^{n-1})$ are open, connected and $\Gamma$-invariant, and $\widetilde{U}_1\cap\widetilde{U}_2=\emptyset$.
\item[(b)]$\widetilde{U}_1$ is $\Gamma$-diffeomorphic to $\mathbb{S}^{m-1}\times\mathbb{B}^n$, $\widetilde{U}_2$ is $\Gamma$-diffeomorphic to $\mathbb{B}^m\times\mathbb{S}^{n-1}$ and their common boundary $\partial\widetilde{U}_1 = \partial\widetilde{U}_2$ is $\Gamma$-diffeomorphic to $\mathbb{S}^{m-1}\times\mathbb{S}^{n-1}$.
\end{itemize}
\end{lemma}

\begin{proof}
Consider the function $\pi:\mathbb{S}^N \to \mathbb{R}^2$ given by $\pi(x,y):=(|x|,|y|)$ where $x\in\mathbb{R}^m$, $y\in \mathbb{R}^n$. Then $\pi$ is continuous and $\Gamma$-invariant. Its image is the arc
$$A:=\{(s,t)\in \mathbb{R}^2:\,s,t\geq 0,\,s^2+t^2=1\}.$$
Set $A_i:=\pi(U_i)$. Note that $A_i$ is nonempty, connected and open in $A$ for $i=1,2$. Moreover, $A_1\cap A_2=\emptyset$ and $\overline{A_1\cup A_2}=A$. Therefore, $A_i$ is an arc and, up to relabeling, $\widetilde{A}_1:=A_1\cup\{(1,0)\}$ and $\widetilde{A}_2:=A_2\cup\{(0,1)\}$ are connected and open in $A$, and $A\smallsetminus (\widetilde{A}_1\cup \widetilde{A}_2) = \{(s_0,t_0)\}$ with $s_0>0$ and $t_0>0$, i.e.,
$$\widetilde{A}_1 = \{(s,t)\in A: s>s_0\}\qquad\text{and}\qquad \widetilde{A}_2 = \{(s,t)\in A: s<s_0\}.$$
Then, $\pi^{-1}(\widetilde{A}_i) = \widetilde{U}_i$ and $\pi^{-1}(s_0,t_0) = \partial\widetilde{U}_1 = \partial\widetilde{U}_2$. Therefore, these sets satisfy (a) and (b).
\end{proof}

\begin{proof}
[Proof of Proposition \ref{prop:tori}] Adding a point at infinity and applying the inverse of the stereographic projection to $\Omega_1$ and $\Omega_2$, we obtain two nonempty $\Gamma$-invariant open connected subsets $U_1$ and $U_1$ of $\mathbb{S}^N$ which, thus, satisfy (a) and (b) of Lemma \ref{lem:tori}.

As $m+n=N+1$ and $m,n\geq 2$, the codimension of the sets $\mathbb{S}^{m-1}\times\{0\}$ and $\{0\}\times\mathbb{S}^{n-1}$ in $\mathbb{R}^N$ is at least $2$. Therefore, $D^{1,2}_0(\Omega_i) = D^{1,2}_0(\widetilde{\Omega}_i)$. Hence, $u_{\infty}$ solves the problem
$$-\Delta u=\mu_{1}|u|^{{2}^{*}-2}u,\text{\qquad}u\in D_{0}^{1,2}(\widetilde{\Omega}_{1}),$$
and, by the maximum principle, $u_{\infty}>0$ in $\widetilde{\Omega}_{1}$. As $\Omega_{1}=\{x\in\mathbb{R}^{N}:u_{\infty}(x)>0\}$, we conclude that $\widetilde{\Omega}_1 = \Omega_1$. Similarly, $\widetilde{\Omega}_2 = \Omega_2$, and the claim is proved.
\end{proof}

\appendix

\section{The energy functional on a plane}

\label{sec:appendix}Consider the function%
\[
e(s,t):=a_{1}s^{2}+a_{2}t^{2}-b_{1}s^{p}-b_{2}t^{p}+ds^{\alpha}t^{\beta}%
\]
in $V:=(0,\infty)\times(0,\infty),$ where $a_{i},b_{i},d>0,$ $p>2,$
$\alpha,\beta>1$ and $\alpha+\beta=p.$ We assume that $(1,1)$ is a critical
point of $e.$ Then, as%
\begin{align*}
\frac{\partial e}{\partial s}(s,t)  &  =2a_{1}s-pb_{1}s^{p-1}+d\alpha
s^{\alpha-1}t^{\beta},\\
\frac{\partial e}{\partial t}(s,t)  &  =2a_{2}t-pb_{2}t^{p-1}+d\beta
s^{\alpha}t^{\beta-1},
\end{align*}
we have that%
\begin{equation}
2a_{1}-pb_{1}+d\alpha=0\text{\qquad and\qquad}2a_{2}-pb_{2}+d\beta=0.
\label{eq:coef}%
\end{equation}

\begin{lemma}
\label{lem:a}There exist $0<r<R<\infty$ and $\delta>0$ such that%
\begin{align*}
\delta &  \leq\frac{\partial e}{\partial s}(r,t)\text{\quad and\quad}%
\frac{\partial e}{\partial s}(R,t)\leq-1\text{\qquad for all }t\in\left[
r,R\right]  ,\\
\delta &  \leq\frac{\partial e}{\partial t}(s,r)\text{\quad and\quad}%
\frac{\partial e}{\partial t}(s,R)\leq-1\text{\qquad for all }s\in\left[
r,R\right]  ,
\end{align*}
every critical point of $e$ in $V$ lies in the interior of $Q:=\left[
r,R\right]  \times\left[  r,R\right]  ,$ and $\sup_{V}e=\max_{Q}e.$
\end{lemma}

\begin{proof}
Let $t=\tau s$ with $\tau\in\left[  0,1\right]  .$ Then, from (\ref{eq:coef})
we get that%
\begin{align*}
\frac{\partial e}{\partial s}(s,t)  &  =2a_{1}s-\left(  pb_{1}-d\alpha
\tau^{\beta}\right)  s^{p-1}\\
&  \leq2a_{1}s-\left(  pb_{1}-d\alpha\right)  s^{p-1}\leq-1\text{\quad for all
}s\in\lbrack R_{1},\infty).
\end{align*}
Similarly, if $s=\tau t$ with $\tau\in\left[  0,1\right]  ,$ we have that
\[
\frac{\partial e}{\partial t}(s,t)\leq-1\text{\quad for all }t\in\lbrack
R_{2},\infty).
\]
On the other hand, as%
\[
\frac{\partial e}{\partial s}(s,t)\geq2a_{1}s-pb_{1}s^{p-1}\text{\quad
and\quad}\frac{\partial e}{\partial t}(s,t)\geq2a_{2}t-pb_{2}t^{p-1},
\]
there exist $r,\delta>0$ such that%
\begin{align*}
\frac{\partial e}{\partial s}(s,t)  &  >0\text{ \ if }s\in(0,r]\text{\quad
and\quad}\frac{\partial e}{\partial s}(r,t)\geq\delta\text{\quad for all }%
t\in(0,\infty),\\
\frac{\partial e}{\partial t}(s,t)  &  >0\text{ \ if }t\in(0,r]\text{\quad
and\quad}\frac{\partial e}{\partial t}(s,r)\geq\delta\text{\quad for all }%
s\in(0,\infty).
\end{align*}
Setting $R:=\max\{R_{1},R_{2}\},$ we obtain our claim.
\end{proof}

\begin{lemma}
\label{lem:unique}If every critical point of $e$ in $V$ is a strict local
maximum, then $(1,1)$ is the only critical point of $e$ in $V$ and it is a
global maximum.
\end{lemma}

\begin{proof}
Let $Q$ be as in Lemma \ref{lem:a}. Then, $Q$ is strictly positively invariant
under the (positive) gradient flow of $e$, so this flow defines a map%
\[
\eta:\mathbb{R}\times Q\rightarrow Q.
\]
As the critical points of $e$ in $V$ are contained in the interior of $Q$ and
they are isolated, there are finitely many of them, $\xi_{1},...,\xi_{m},$
and, since each of them is a strict local maximum, we may choose
$\varepsilon>0$ such that $\overline{B}_{\varepsilon}(\xi_{i}):=\{x\in
\mathbb{R}^{2}:\left\vert x-\xi_{i}\right\vert \leq\varepsilon\}\subset Q,$
$\overline{B}_{\varepsilon}(\xi_{i})\cap\overline{B}_{\varepsilon}(\xi
_{j})=\emptyset$ if $i\neq j,$ and $\overline{B}_{\varepsilon}(\xi_{i})$ is
strictly positively invariant under the gradient flow of $e.$ Set
$\Theta:=\overline{B}_{\varepsilon}(\xi_{1})\cup\cdots\cup\overline
{B}_{\varepsilon}(\xi_{m}).$ Then, the entrance time function $T_{\Theta
}:Q\rightarrow\mathbb{R},$ defined by%
\[
T_{\Theta}(x):=\inf\{\tau\geq0:\eta(\tau,x)\in\Theta\},
\]
is continuous. Therefore, the map $\pi:Q\rightarrow\Theta$ given by%
\[
\pi(x):=\eta(T_{\Theta}(x),x)
\]
is also continuous and it is surjective. As $Q$ is connected, $\Theta$ cannot
have more than one component. Therefore, $e$ has only one critical point and
it is a global maximum.
\end{proof}

\bigskip
\begin{acknowledgement}
We wish to thank the anonymous referee for his/her careful reading and valuable comments.
\end{acknowledgement}

\end{document}